\documentclass[11pt,english, a4paper]{article}

\usepackage[english]{babel}
\usepackage[utf8x]{inputenc}
\usepackage[T1]{fontenc}
\usepackage{authblk}

\usepackage[a4paper,top=3cm,bottom=3cm,left=3cm,right=3cm,marginparwidth=1.75cm]{geometry}

\usepackage[colorlinks=true, allcolors=blue, hypertexnames=false]{hyperref}
\usepackage{amsmath,empheq}
\usepackage{graphicx}
\usepackage[colorinlistoftodos]{todonotes}
\usepackage{amsmath,mathtools}
\usepackage{upgreek}
\usepackage{amssymb}
\usepackage{amsfonts}
\usepackage{amsthm}
\usepackage{enumerate}
\usepackage{mathrsfs}
\usepackage{fontenc}

\numberwithin{equation}{section}
\theoremstyle{plain}
\newtheorem{thm}{Theorem}
\newtheorem{lem}{Lemma}

\newtheorem{assu}{Assumption}
\newenvironment{thmbis}[1]
 {\renewcommand{\thethm}{\ref{#1}$'$}%
 \addtocounter{thm}{-1}%
 \begin{thm}}
 {\end{thm}}

\theoremstyle{definition}

\theoremstyle{remark}
\newtheorem*{rem}{Remark}

\newcommand{\bR}{\mathbb {R}}
\newcommand{\transpose}{^{\operatorname{T}}}
\newcommand{\rmd}{\mathrm{d}}

\title{Convergence of the Deep BSDE Method \\for Coupled FBSDEs}
\author{Jiequn Han \thanks{jiequnh@princeton.edu}}
\author[2]{Jihao Long \thanks{jihaol@pku.edu.cn}}
\affil[1]{Department of Mathematics, Princeton University, Princeton, NJ 08544, USA}
\affil[2]{School of Mathematical Sciences, Peking University, Beijing 100871, China}

\date{}

\begin{document}
\maketitle

\begin{abstract}
The recently proposed numerical algorithm, deep BSDE method, has shown remarkable performance in solving high-dimensional forward-backward stochastic differential equations (FBSDEs) and parabolic partial differential equations (PDEs).
This article lays a theoretical foundation for the deep BSDE method in the general case of coupled FBSDEs. In particular, a posteriori error estimation of the solution is provided and it is proved that the error converges to zero given the universal approximation capability of neural networks.
Numerical results are presented to demonstrate the accuracy of the analyzed algorithm in solving high-dimensional coupled FBSDEs.
\end{abstract}

\section{Introduction}
Forward-backward stochastic differential equations (FBSDEs) and partial differential equations (PDEs) of parabolic type have found numerous applications in stochastic control, finance, physics, etc., as a ubiquitous modeling tool. In most situations encountered in practice the equations cannot be solved analytically but require certain numerical algorithms to provide approximate solutions. On the one hand, the dominant choices of numerical algorithms for PDEs are mesh-based methods, such as finite differences, finite elements, etc. On the other hand, FBSDEs can be tackled directly through probabilistic means, with appropriate methods for the approximation of conditional expectation. Since these two kinds of equations are intimately connected through the nonlinear Feynman--Kac formula~\cite{pardoux1992backward}, the algorithms designed for one kind of equation can often be used to solve another one.

However, the aforementioned numerical algorithms become more and more difficult, if not impossible, when the dimension increases.
They are doomed to run into the so-called ``curse of dimensionality'' \cite{bellman1957dynamic} when the dimension is high, namely, the computational complexity grows exponentially as the dimension grows. 
The classical mesh-based algorithms for PDEs require a mesh of size $O(N^d)$. 
The simulation of FBSDEs faces a similar difficulty in the general nonlinear cases, due to the need to compute conditional expectation in high dimension.
The conventional methods, including the least squares regression~\cite{bender2012least}, Malliavin approach~\cite{bouchard2004malliavin}, and kernel regression~\cite{bouchard2004discrete}, are all of exponential complexity. There are a limited number of cases where practical high-dimensional algorithms are available. For example, in the linear case, Feynman--Kac formula and Monte Carlo simulation together provide an efficient approach to solving PDEs and associated BSDEs numerically. 
In addition, methods based on the branching diffusion process \cite{henry2012counterparty, henry2019branching} and  multilevel Picard iteration \cite{hutzenthaler2016multilevel,weinan2019multilevel,hutzenthaler2018overcoming} overcome the curse of dimensionality in their considered settings. 
We refer~\cite{weinan2019multilevel} for the detailed discussion on the complexity of the algorithms mentioned above.
Overall there is no numerical algorithm in literature so far proved to overcome the curse of dimensionality for general quasilinear parabolic PDEs and the corresponding FBSDEs.

A recently developed algorithm, called the deep BSDE method~\cite{han2018solving,weinan2017deep}, has shown astonishing power in solving general high-dimensional FBSDEs and parabolic PDEs~\cite{beck2017machine,han2019deep,han2020solving}. In contrast to conventional methods, the deep BSDE method employs neural networks to approximate unknown gradients and reformulates the original equation-solving problem into a stochastic optimization problem. Thanks to the universal approximation capability and parsimonious parameterization of neural networks, in practice the objective function can be effectively optimized in high-dimensional cases, and the function values of interests are obtained quite accurately.

The deep BSDE method was initially proposed for decoupled FBSDEs.
In this paper, we extend the method to deal with coupled FBSDEs and a broader class of quasilinear parabolic PDEs. Furthermore, we present an error analysis of the proposed scheme, including decoupled FBSDEs as a special case. 
Our theoretical result consists of two theorems.  Theorem~\ref{thm:main1} provides a posteriori error estimation of the deep BSDE method. As long as the objective function is optimized to be close to zero under fine time discretization, the approximate solution is close to the true solution. In other words, in practice, the accuracy of the numerical solution is effectively indicated by the value of the objective function.
Theorem~\ref{thm:main2} shows that such a situation is attainable, by relating the infimum of the objective function to the expression ability of neural networks. As an implication of the universal approximation property (in the $L^2$ sense), there exist neural networks with suitable parameters such that the obtained numerical solution is approximately accurate.
To the best of our knowledge, this is the first theoretical result of the deep BSDE method for solving FBSDEs and parabolic PDEs. 
Although our numerical algorithm is based on neural networks, the theoretical result provided here is equally applicable to the algorithms based on other forms of function approximations.

The article is organized as follows. In section 2, we precisely state our numerical scheme for coupled FBSDEs and quasilinear parabolic PDEs and give the main theoretical results of the proposed numerical scheme.
In section 3, the basic assumptions and some useful results from the literature are given for later use. The proofs of the two main theorems are provided in section 4 and section 5, respectively.
Some numerical experiments with the proposed scheme are presented in section 6.

\section{A Numerical Scheme for Coupled FBSDEs and Main Results}

Let $T \in (0, +\infty)$ be the terminal time, $(\Omega,\mathbb{F},\{\mathcal{F}_t\}_{0 \le t \le T}, \mathbb{P})$ be a  filtered probability space equipped with a $d$-dimensional standard Brownian motion $\{W_t\}_{0 \le t \le T}$ starting from $0$. $\xi$ is a square-integrable random variable independent of $\{W_t\}_{0 \le t \le T}$. We use the same notation $(\Omega,\mathbb{F},\{\mathcal{F}_t\}_{0 \le t \le T}, \mathbb{P})$ to denote the filtered probability space generated by $\{W_t + \xi\}_{0 \le t \le T}$.
The notation $|x|$ denotes the Euclidean norm of a vector $x$ and $\|A\| = \sqrt{\mathrm{trace}(A\transpose A)}$ denotes the Frobenius norm of a matrix $A$.

Consider the following coupled FBSDEs
\begin{empheq}[left=\empheqlbrace]{align}
    &X_t = \xi + \int_{0}^{t}b(s,X_s,Y_s)\, \mathrm{d}s + \int_{0}^{t}\sigma(s,X_s,Y_s)\, \mathrm{d}W_s, \label{eq:FBSDE_forward} \\
    &Y_t = g(X_T) + \int_{t}^{T}f(s,X_s,Y_s,Z_s)\, \mathrm{d}s - \int_{t}^{T}(Z_s)\transpose\, \mathrm{d}W_s, \label{eq:FBSDE_backward}
\end{empheq}
in which $X_t$ takes values in $\mathbb{R}^m$, $Y_t$ takes values in $\mathbb{R}$, and $Z_t$ takes values in $\mathbb{R}^d$. 
Here we assume $Y_t$ to be one-dimensional to simplify the presentation. The result can be extended without any difficulty to the case where $Y_t$ is multi-dimensional.  We say $(X_t, Y_t, Z_t)$ is a solution of the above FBSDEs, if all its components are $\mathcal{F}_{t}$-adapted and square-integrable, together satisfying equations \eqref{eq:FBSDE_forward}\eqref{eq:FBSDE_backward}.

Solving coupled FBSDEs numerically is more difficult than solving decoupled FBSDEs. Except the Picard iteration method developed in~\cite{bender2008time}, most methods exploit the relation to quasilinear parabolic PDEs via the four-time-step-scheme in \cite{ma1994solving}. This type of methods suffers from high dimensionality due to spatial discretization of PDEs. In contrast, our strategy, starting from simulating the coupled FBSDEs directly, is a new purely probabilistic scheme. To state the numerical algorithm precisely, we consider a partition of the time interval $[0,T]$, $\pi: 0 = t_0 < t_1 < \dots < t_N = T$ with $h = T/N$ and $t_i = ih$. Let $\Delta W_i \coloneqq W_{t_{i+1}} - W_{t_i}$ for $i = 0, 1, \dots, N-1$.
Inspired by the nonlinear Feynman--Kac formula that will be introduced below, we view $Y_0$ as a function of $X_0$ and view $Z_t$ as a function of $X_t$ and $Y_t$. 
Equipped with this viewpoint, our goal becomes finding appropriate functions $\mu_0^\pi: \mathbb{R}^m \rightarrow \bR$ and $\phi_i^\pi:\mathbb{R}^m\times \mathbb{R} \rightarrow \mathbb{R}^d$ for $i = 0, 1,\dots,N-1$ such that $\mu_0^\pi(\xi)$ and $\phi_i^\pi(X_{t_i}^\pi, Y_{t_i}^\pi)$ can serve as good surrogates of $Y_0$ and $Z_{t_i}$, respectively. To this end, we consider the classical Euler scheme
\begin{align}
\begin{dcases}
    X_0^{\pi} = \xi, \quad Y_0^{\pi} = \mu_0^\pi(\xi), \\
    X_{t_{i+1}}^{\pi} = X_{t_i}^{\pi} + b(t_i,X_{t_i}^{\pi},Y_{t_i}^{\pi})h + \sigma(t_i,X_{t_i}^{\pi},Y_{t_i}^{\pi})\Delta W_i, \\
    Z_{t_i}^\pi = \phi_i^\pi(X_{t_i}^\pi,Y_{t_i}^\pi), \\
    Y_{t_{i+1}}^{\pi} = Y_{t_i}^{\pi} - f(t_i,X_{t_i}^{\pi},Y_{t_i}^\pi,Z_{t_i}^\pi)h + (Z_{t_i}^\pi)\transpose\Delta W_i.
\end{dcases}
\label{eq:FB_system1}
\end{align}
Without loss of clarity, here we use the notation $X_0^{\pi}$ as $X_{t_0}^{\pi}$, $X_T^{\pi}$ as $X_{t_N}^{\pi}$, etc.

Following the spirit of the deep BSDE method, we employ a stochastic optimizer to solve the following stochastic optimization problem 
\begin{align}
    &\inf_{\mu_0^{\pi} \in \mathcal{N}'_0, \phi_i^\pi \in \mathcal{N}_i} F(\mu_0^{\pi},\phi_0^\pi,\dots,\phi_{N-1}^\pi) \coloneqq E|g(X_T^{\pi}) - Y_T^{\pi}|^2, 
\label{eq:disc_objective}
\end{align}
where $\mathcal{N}'_0$ and $\mathcal{N}_i~(0\leq i \leq {N-1})$ are parametric function spaces generated by neural networks. 
To see intuitively where the objective function~\eqref{eq:disc_objective} comes from, we consider the following variational problem:
\begin{align}
&\inf_{Y_0,\{Z_t\}_{0\le t \le T}} E|g(X_T) - Y_T|^2, \label{eq:objective}\\
&s.t.~~ X_t = \xi + \int_{0}^{t}b(s,X_s,Y_s)\, \mathrm{d}s + \int_{0}^{t}\sigma(s,X_s,Y_s)\, \mathrm{d}W_s, \notag \\
 & \qquad Y_t = Y_0 - \int_{0}^{t}f(s,X_s,Y_s,Z_s)\, \mathrm{d}s + \int_{0}^{t}(Z_s)\transpose\, \mathrm{d}W_s,  \notag
\end{align}
where $Y_0$ is $\mathcal{F}_0$-measurable and square-integrable, and $Z_t$ is a $\mathcal{F}_t$-adapted square-integrable process. 
The solution of the FBSDEs~\eqref{eq:FBSDE_forward}\eqref{eq:FBSDE_backward} is a minimizer of the above problem since the loss function attains zero when it is evaluated at the solution.
In addition, the wellposedness of the FBSDEs (under some regularity conditions) ensures the existence and uniqueness of the minimizer. 
Therefore, we expect~\eqref{eq:disc_objective}, as a discretized counterpart of~\eqref{eq:objective}, defines a benign optimization problem and the associated near-optimal solution
provides us a good approximate solution of the original FBSDEs. The reason we do not represent $Z_{t_i}$ as a function of $X_{t_i}$ only is that the process $\{X_{t_i}^\pi\}_{0\le i \le N}$ is not Markovian, while the process $\{X_{t_i}^\pi,Y_{t_i}^\pi\}_{0 \le i \le N}$ is Markovian, which facilitates our analysis considerably. If $b$ and $\sigma$ are both independent of $Y$, then the FBSDEs~\eqref{eq:FBSDE_forward}\eqref{eq:FBSDE_backward} are decoupled, we can take $\phi_i^\pi$ as a function of $X_{t_i}^\pi$ only, as the numerical scheme introduced in \cite{han2018solving,weinan2017deep}.

Our two main theorems regarding the deep BSDE method are the following, mainly on the justification and property of the objective function~\eqref{eq:disc_objective} in the general coupled case, regardless of the specific choice of parametric function spaces.
An important assumption for the two theorems is the so-called weak coupling or monotonicity condition, which will be explained in detail in section 3.
The precise statement of the theorems can be found in Theorem \ref{thm:main1_detail} (section 4) and Theorem \ref{thm:main2_detail} (section 5), respectively.
\begin{thm}
\label{thm:main1}
Under some assumptions, there exists a constant C, independent of h, d, and m, such that for sufficiently small h,
\begin{equation}
\label{eq:main1}
    \sup_{ t \in [0,T]} (E|X_t - \hat{X}_t^{\pi}|^2 + E|Y_t - \hat{Y}_t^{\pi}|^2) + \int_{0 }^{T} E|Z_t - \hat{Z}_t^\pi|^2 \, \mathrm{d}t \le C[h +  E|g(X_{T}^\pi) - Y_T^{\pi}|^2],
\end{equation}
where $\hat{X}_t^{\pi} = X_{t_i}^\pi$, $\hat{Y}_t^{\pi} = Y_{t_i}^\pi$, $\hat{Z}_t^{\pi} = Z_{t_i}^\pi$ for $t \in [t_i,t_{i+1})$.
\end{thm}
 
\begin{thm}
\label{thm:main2}
Under some assumptions, there exists a constant C, independent of h, d and m, such that for sufficiently small h,
\begin{align*}
    &\inf_{\mu_0^\pi \in \mathcal{N}'_0 , \phi_i^\pi \in \mathcal{N}_i} E|g(X_T^\pi) - Y_T^{\pi}|^2 \notag \\
    \le~& C\Big\{h +\inf_{\mu_0^\pi \in \mathcal{N}'_0, \phi_i^\pi \in \mathcal{N}_i} \big[ E|Y_0 - \mu_0^\pi(\xi)|^2 \\
    &\quad\quad\quad\quad\quad\quad\quad\quad~~ + \sum_{i = 0}^{N-1}E|E[\tilde{Z}_{t_i}|X_{t_i}^\pi,Y_{t_i}^\pi] - \phi_i^\pi(X_{t_i}^\pi,Y_{t_i}^\pi)|^2h\big] \Big\},
\end{align*}
where $\tilde{Z}_{t_i} =h^{-1}E[\int_{t_i}^{t_{i+1}}Z_t\,\rmd t|\mathcal{F}_{t_i}]$. If $b$ and $\sigma$ are independent of $Y$, the term $E[\tilde{Z}_{t_i}|X_{t_i}^\pi,Y_{t_i}^\pi]$ can be replaced with $E[\tilde{Z}_{t_i}|X_{t_i}^\pi]$. 
\end{thm}

Briefly speaking, Theorem \ref{thm:main1} states that the simulation error (left side of equation~\eqref{eq:main1}) can be bounded through the value of the objective function~\eqref{eq:disc_objective}. 
To the best of our knowledge, this is the first result for the error estimation of the coupled FBSDEs, concerning both time discretization error  and terminal distance.
Theorem \ref{thm:main2} states that the optimal value of the objective function can be small if the approximation capability of the parametric function spaces ($\mathcal{N}'_0$ and $\mathcal{N}_i$ above) is high. 
Neural networks are a promising candidate for such a requirement, especially in high-dimensional problems.
There are numerous results, dating back to the 90s~(see, e.g., \cite{cybenko1989approximation,funahashi1989approximate,hornik1989multilayer,barron1993universal,arora2016understanding,eldan2016power,cohen2016expressive,mhaskar2016deep,bolcskei2017optimal,liang2017deep}), in regard to the universal approximation and complexity of neural networks.
There are also some recent analysis~\cite{grohs2018proof,jentzen2018proof,berner2018analysis,hutzenthaler2019proof} on approximating the solutions of certain parabolic partial differential equations with neural networks.
However, the problem is still far from resolved. Theorem \ref{thm:main2} implies that if the involved conditional expectations can be approximated by neural networks whose numbers of parameters growing at most polynomially both in the dimension and the reciprocal of the required accuracy, then the solutions of the considered FBSDEs can be represented in practice without the curse of dimensionality. Under what conditions this assumption is true is beyond the scope of this work and remains for further investigation.

The above-mentioned scheme in \eqref{eq:FB_system1}\eqref{eq:disc_objective} is for solving FBSDEs. The so-called nonlinear Feynman--Kac formula, connecting FBSDEs with the quasilinear parabolic PDEs, provides an approach to numerically solve quasilinear parabolic PDEs \eqref{eq:pde} below through the same scheme. 
We recall a concrete version of the nonlinear Feynman--Kac formula in Theorem \ref{thm:Feynman-Kac} below and refer interested readers to e.g., \cite{ma2007forward} for more details. According to this formula, the term $E|Y_0 - Y_0^\pi|^2$ can be interpreted as $E|u(0,\xi) - \mu_0^\pi(\xi)|^2$. Therefore, we can choose the random variable $\xi$ with a delta distribution, a uniform distribution in a bounded region, or any other distribution we are interested in. After solving the optimization problem, we obtain $\mu_0^\pi(\xi)$ as an approximation of $u(0,\xi)$. See \cite{han2018solving,weinan2017deep} for more details.

\begin{thm}\label{thm:Feynman-Kac}
Assume
\begin{enumerate}[\indent 1.]
\item $m = d$ and $b(t,x,y)$, $\sigma(t,x,y)$, $f(t,x,y,z)$ are smooth functions with bounded first-order derivatives with respect to $x, y, z$.
\item There exist a positive continuous function $\nu$ and a constant $\mu$, satisfying that
\begin{align*}
&\nu(|y|)\mathbf{I} \le \sigma\sigma\transpose(t,x,y) \le \mu \mathbf{I}, \\
&|b(t,x,0)| + |f(t,x,0,z)| \le \mu.
\end{align*}
\item There exists a constant $\alpha \in (0,1)$ such that $g$ is bounded in the H\"older space $C^{2,\alpha}(\bR^m)$.
\end{enumerate}
Then the following quasilinear PDE has a unique classical solution $u(t,x)$ that is bounded with bounded $u_t$, $\nabla_x u$, and $\nabla^2_x u$,
\begin{align}
\begin{dcases}
u_t + \frac{1}{2}\text{trace}(\sigma\sigma\transpose(t,x,u)\nabla^2_x u)   \\
\hphantom{u_t} +~b\transpose(t,x,u) \nabla_x u + f(t,x,u,\sigma\transpose(t,x,u) \nabla_x u) = 0, \\
u(T,x) = g(x). \label{eq:pde}
\end{dcases}
\end{align}
The associated FBSDEs~\eqref{eq:FBSDE_forward}\eqref{eq:FBSDE_backward} have a unique solution $(X_t, Y_t, Z_t)$ with $Y_t = u(t,X_t)$, $Z_t = \sigma\transpose(t,X_t,u(t,X_t))\nabla_x u(t,X_t)$, and $X_t$ is the solution of the following SDE
\begin{equation*}
X_t = \xi + \int_{0}^{t}b(s,X_s,u(s,X_s))\, \mathrm{d}s + \int_{0}^{t}\sigma(s,X_s,u(s,X_s))\, \mathrm{d}W_s.
\end{equation*}
\end{thm}

\begin{rem}
The statement regarding FBSDEs ~\eqref{eq:FBSDE_forward}\eqref{eq:FBSDE_backward} in Theorem \ref{thm:Feynman-Kac} is developed through a PDE-based argument, which requires $m = d$, uniform ellipticity of $\sigma$, and high-order smoothness of $b,\sigma,f$, and $g$. An analogous result through probabilistic
argument is given below in Theorem \ref{thm4} (point 4). In that case, we only need the Lipschitz condition for all of the involved functions, in addition to some weak coupling or monotonicity conditions demonstrated in Assumption \ref{assu3}. Note that the Lipschitz condition alone does not guarantee the existence of a solution to the coupled FBSDEs, even in the situation when $b,f,\sigma$ are linear (see \cite{bender2008time,ma2007forward} for a concrete counterexample).
\end{rem}
\medskip

\begin{rem}
Theorem \ref{thm:Feynman-Kac} also implies that the assumption that the drift function $b$ only depends on $x, y$ is general. If $b$ depends on $z$ as well, one can  move the associated term in~\eqref{eq:pde} into the nonlinearity $f$ and apply the nonlinear Feynman--Kac formula back to obtain an equivalent system of coupled FBSDEs, in which the new drift function is independent of $z$.
\end{rem}

\section{Preliminaries}
In this section, we introduce our assumptions and two useful results in \cite{bender2008time}.
We use the notation $\Delta x = x_1 - x_2$, $\Delta y = y_1 - y_2$, $\Delta z = z_1 - z_2$.  
\begin{assu}\label{assu1}
\begin{enumerate}[(i)]
\item There exist (possibly negative) constants $k_b$, $k_f$ such that
\begin{align*}
[b(t,x_1,y) - b(t,x_2,y)]\transpose\Delta x &\le k_b |\Delta x|^2, \\
[f(t,x,y_1,z) - f(t,x,y_2,z)]\Delta y &\le k_f |\Delta y|^2.
\end{align*}

\item b, $\sigma$, f, g are uniformly Lipschitz continuous with respect to (x,y,z). In particular, there are non-negative constants K, $b_y$, $\sigma_x$, $\sigma_y$, $f_x$, $f_z$, and $g_x$ such that
\begin{align*}
|b(t,x_1,y_1) - b(t,x_2,y_2)|^2 &\le K|\Delta x|^2 + b_y|\Delta y|^2, \\
\|\sigma(t,x_1,y_1) - \sigma(t,x_2,y_2)\|^2 &\le \sigma_x|\Delta x|^2 +  \sigma_y|\Delta y|^2, \\
|f(t,x_1,y_1,z_1) - f(t,x_2,y_2,z_2)|^2 &\le f_x|\Delta x|^2 + K|\Delta y|^2 + f_z|\Delta z|^2, \\
|g(x_1) - g(x_2)|^2 &\le g_x|\Delta x|^2.
\end{align*}
\item $b(t,0,0)$, $f(t,0,0,0)$, and $\sigma(t,0,0)$ are bounded. In particular, there are constants $b_0$, $\sigma_0$, $f_0$, and $g_0$ such that
\begin{align*}
|b(t,x,y)|^2 &\le b_0 + K|x|^2 + b_y|y|^2, \\
\|\sigma(t,x,y)\|^2 &\le \sigma_0 + \sigma_x|x|^2 + \sigma_y|y|^2, \\
|f(t,x,y,z)|^2 &\le f_0 + f_x|x|^2 + K|y|^2 + f_z|z|^2, \\
|g(x)|^2 &\le g_0 + g_x|x|^2. 
\end{align*}
\end{enumerate}
We note here $b_y$ et al. are all constants, not partial derivatives. 
For convenience, we use $\mathscr{L}$ to denote the set of all the constants mentioned above and assume K is the upper bound of $\mathscr{L}$.
\end{assu}

\begin{assu}\label{assu2}
$b, \sigma, f$ are uniformly H\"older-$\frac{1}{2}$ continuous with respect to $t$. We assume the same constant K to be the upper bound of the square of the H\"older constants as well.
\end{assu}

\begin{assu}\label{assu3}
One of the following five cases holds:
\begin{enumerate}[\indent 1.]
\item Small time duration, that is, T is small.
\item  Weak coupling of Y into the forward SDE~\eqref{eq:FBSDE_forward}, that is, $b_y$ and $\sigma_y$ are small. In particular, if $b_y = \sigma
_y = 0$, then the forward equation does not depend on the backward one and, thus, equations~\eqref{eq:FBSDE_forward}\eqref{eq:FBSDE_backward} are decoupled.
\item Weak coupling of X into the backward SDE~\eqref{eq:FBSDE_backward}, that is, $f_x$ and $g_x$ are small. In particular, if $f_x = g_x = 0$, then the backward equation does not depend on the forward one and, thus, equations~\eqref{eq:FBSDE_forward}\eqref{eq:FBSDE_backward} are also decoupled. In fact, in this case, Z = 0 and~\eqref{eq:FBSDE_backward} reduces to an ODE.
\item f is strongly decreasing in y, that is, $k_f$ is very negative.
\item b is strongly decreasing in x, that is, $k_b$ is very negative.
\end{enumerate}
\end{assu}
The assumptions stated above are usually called weak coupling and monotonicity conditions in literature~\cite{bender2008time,antonelli1993backward,pardoux1999forward}.
To make it more precise, we define
\begin{align*}
&L_0 = [b_y +\sigma_y ][g_x +f_x T]Te^{[b_y +\sigma_y][g_x +f_x T]T+[2k_b +2k_f +2+\sigma_x +f_z ]T}, \\
&L_1 = [g_x +f_x T][e^{[b_y +\sigma_y][g_x +f_x T]T+[2k_b +2k_f +2+\sigma_x +f_z ]T + 1} \vee 1], \\
&\Gamma_0(x) = \frac{e^x - 1}{x},  ~~(x>0), \\
&\Gamma_1(x,y) = \sup_{0 < \theta < 1}\theta e^{\theta x}\Gamma_0(y), \\
&c = \inf_{\lambda_1 > 0}\Big\{[e^{[2k_b +1+\sigma_x +[b_y +\sigma_y ]L_1]T} \vee 1](1 + \lambda_1^{-1})[b_y+\sigma_y]T \notag \\
&\qquad\qquad\quad~ \times[g_x\Gamma_1([2k_f+1 + f_z]T, [2k_b+1+\sigma_x+(1+\lambda_1)[b_y+\sigma_y]L_1]T)   \notag \\
&\qquad\qquad~ + f_x T\Gamma_0([2k_f+1+f_z]T) \\
&\qquad\qquad\quad~ \times \Gamma_0 ( 2k_b +1+\sigma_x +(1+\lambda_1 )[b_y +\sigma_y]L_1]T)\Big\}.
\end{align*}
Then, a specific quantitative form of the above five conditions can be summarized as:
\begin{equation}
\label{eq:assu3_ineq}
L_0 < e^{-1} \text{~~and~~} c < 1.
\end{equation}
In other words, if any of the five conditions of the weak coupling and monotonicity conditions holds to certain extent, the two inequalities in~\eqref{eq:assu3_ineq} hold. Below, we refer to ~\eqref{eq:assu3_ineq} as Assumption~\ref{assu3} and the five general qualitative conditions described above as the weak coupling and monotonicity conditions.

The above three assumptions are basic assumptions in \cite{bender2008time}, which we need in order to use the results from \cite{bender2008time}, as stated in Theorems \ref{thm4} and \ref{thm5} below.
Theorem~\ref{thm4} gives the connections between coupled FBSDEs and quasilinear parabolic PDEs under weaker conditions. Theorem~\ref{thm5} provides the convergence of the implicit scheme for coupled FBSDEs. Our work primarily uses the same set of assumptions except that we assume some further quantitative restrictions related to the weak coupling and monotonicity conditions, which will be transparent through the extra constants we define in proofs. Our aim is to provide explicit conditions on which our results hold and more clearly present the relationship between these constants and the error estimates. As will be seen in the proof, roughly speaking, the weaker the coupling (resp., the stronger the monotonicity, the smaller the time horizon) is, the easier the condition is satisfied, and the smaller the constant $C$ related with error estimates are.

\begin{thm}\label{thm4}
Under Assumptions \ref{assu1}, \ref{assu2}, and \ref{assu3}, there exists a function u: $\bR \times \bR^m \rightarrow \bR$ that satisfies the following statements.
\begin{enumerate}[\indent 1.]
\item $|u(t,x_1) - u(t,x_2)|^2 \le L_1|x_1-x_2|^2$.
\item $|u(s,x) - u(t,x)|^2 \le C(1+|x|^2)|s-t|$ with some constant C depending on $\mathscr{L}$ and $T$.

\item u is a viscosity solution of the PDE \eqref{eq:pde}.

\item The FBSDEs \eqref{eq:FBSDE_forward}\eqref{eq:FBSDE_backward} have a unique solution $(X_t,Y_t,Z_t)$ and $Y_t = u(t,X_t)$. Thus, $(X_t,Y_t,Z_t)$ satisfies decoupled FBSDEs
\begin{empheq}[left=\empheqlbrace]{align*}
    X_t &= \xi + \int_{0}^{t}b(s,X_s,u(s,X_s))\, \mathrm{d}s + \int_{0}^{t}\sigma(s,X_s,u(s,X_s))\, \mathrm{d}W_s, \\
Y_t &= g(X_T) + \int_{t}^{T}f(s,X_s,Y_s,Z_s) \, \mathrm{d}s - \int_{t}^{T}(Z_s)\transpose \, \mathrm{d}W_s.
\end{empheq}
\end{enumerate}
Furthermore, the solution of the FBSDEs satisfies the path regularity with some constant C depending on $\mathscr{L}$ and T
\begin{equation}
\sup_{t\in[0,T]} (E|X_t - \tilde{X}_t|^2 + E|Y_t - \tilde{Y}_t|^2) + \int_{0}^{T} E|Z_t - \tilde{Z}_{t}|^2\,\rmd t \le C(1 + E|\xi|^2)h, \label{eq:path_regularity}
\end{equation}
in which $\tilde{X}_t = X_{t_i}$, $\tilde{Y}_t = Y_{t_i}$, $\tilde{Z}_t =h^{-1}E[\int_{t_i}^{t_{i+1}}Z_t \,\rmd t|\mathcal{F}_{t_i}]$ for $t \in [t_i,t_{i+1})$. If $Z_t$ is c\`adl\`ag, we can replace $h^{-1}E[\int_{t_i}^{t_{i+1}}Z_t\,\rmd t|\mathcal{F}_{t_i}]$ with $Z_{t_i}$. 

\end{thm}
\begin{rem}
Several conditions can guarantee $Z_t$ admits a c\`adl\`ag version, such as $m = d$ and $\sigma\sigma\transpose \ge \delta I$ with some $\delta > 0$, see e.g., \cite{zhang2004numerical}. 
\end{rem}
\medskip

\begin{thm}
\label{thm5}
Under Assumptions \ref{assu1}, \ref{assu2}, and \ref{assu3}, for sufficiently small h, the following discrete-time equation ($0 \le i \le N-1$)
\begin{align}
\begin{dcases}
    \overline{X}_0^{\pi} = \xi,   \\
    \overline{X}_{t_{i+1}}^{\pi} = \overline{X}_{t_i}^{\pi} + b(t_i,\overline{X}_{t_i}^{\pi},\overline{Y}_{t_i}^{\pi})h + \sigma(t_i,\overline{X}_{t_i}^{\pi},\overline{Y}_{t_i}^{\pi})\Delta W_i, \\
    \overline{Y}_T^{\pi} = g(\overline{X}_T^{\pi}),  \\
    \overline{Z}_{t_i}^{\pi} = \frac{1}{h}E[\overline{Y}_{t_{i+1}}^{\pi}\Delta W_i|\mathcal{F}_{t_i}],  \\
    \overline{Y}_{t_i}^{\pi} = E[\overline{Y}_{t_{i+1}}^\pi + f(t_i,\overline{X}_{t_i}^{\pi},\overline{Y}_{t_i}^\pi,\overline{Z}_{t_{i}}^{\pi})h|\mathcal{F}_{t_i}],
\end{dcases}
\label{eq:discrete_FBSDE}
\end{align}
has a solution $(\overline{X}_{t_i}^\pi, \overline{Y}_{t_i}^\pi, \overline{Z}_{t_i}^\pi)$ such that $\overline{X}_{t_i}^\pi \in L^2(\Omega,\mathcal{F}_{t_i},\mathbb{P})$ and
\begin{equation}
\sup_{t\in[0,T]} (E|X_t - \overline{X}_t^\pi|^2 + E|Y_t - \overline{Y}_t^\pi|^2) + \int_{0}^{T} E|Z_t - \overline{Z}_{t}^\pi|^2\, \mathrm{d}t \le C(1 + E|\xi|^2)h, 
\label{eq:discrete_convergece}
\end{equation}
where $\overline{X}_t^\pi = \overline{X}_{t_i}^\pi$, $\overline{Y}_t ^\pi= \overline{Y}_{t_i}^\pi$, $\overline{Z}_t^\pi = \overline{Z}_{t_i}^\pi$ for $t \in [t_i,t_{i+1})$, and C is a constant depending on $\mathscr{L}$ and T.
\end{thm}
\begin{rem}
In \cite{bender2008time}, the above result (existence and convergence) is proved for the explicit scheme, which is formulated as replacing  $f(t_i,\overline{X}_{t_i}^{\pi},\overline{Y}_{t_i}^\pi,\overline{Z}_{t_{i}}^{\pi})$ with $ f(t_i,\overline{X}_{t_i}^{\pi},\overline{Y}_{t_{i+1}}^\pi,\overline{Z}_{t_{i}}^{\pi})$ in the last equation of (\ref{eq:discrete_FBSDE}). 
The same techniques can be used to prove the implicit scheme, as we state in Theorem~\ref{thm5}.
\end{rem}
\medskip

Finally, to make sure the system in (2.3) is well-defined, we restrict our parametric function spaces $\mathcal{N}'_0$ and $\mathcal{N}_i$ as in Assumption \ref{assu4} below. 
Note that neural networks with common activation functions, including ReLU and sigmoid function, satisfy this assumption.
Under Assumption \ref{assu1} and \ref{assu4}, one can easily prove by induction that 
$\{X_{t_i}^\pi\}_{0\le i \le N}$, $\{Y_{t_i}^\pi\}_{0 \le i \le N}$ and $\{Z_{t_i}^\pi\}_{ 0 \le i \le N-1}$ defined in \eqref{eq:FB_system1} are all measurable and square-integrable random variables. 

\begin{assu}\label{assu4}
$\mathcal{N}_0^{'}$ and $\mathcal{N}_i (0 \le i \le N-1)$ are subsets of measurable functions from $\bR^m$ to $\bR$ and $\bR^m\times \bR$ to $\bR^d$ with linear growth, namely, $\mu_0^\pi$ and $\{\phi_i^\pi\}_{0 \le i \le N-1}$ in \eqref{eq:FB_system1} satisfy $|\mu_0^\pi(x)|^2 \le A'_0+ B'_0|x|^2$ and $|\phi_i^\pi(x,y)|^2 \le A_i + B_i|x|^2 + C_i|y|^2$ for $0 \le i \le N-1$. 
\end{assu}

\section{A Posteriori Estimation of the Simulation Error}
We prove Theorem \ref{thm:main1} in this section. Comparing the statements of Theorem~\ref{thm:main1} and Theorem~\ref{thm5}, we wish to bound the differences between $(X_{t_i}^\pi, Y_{t_i}^\pi, Z_{t_i}^\pi)$ and $(\overline{X}_{t_i}^\pi, \overline{Y}_{t_i}^\pi, \overline{Z}_{t_i}^\pi)$
with the objective function $E|g(X_T^\pi) - Y_T^\pi|^2$.
Recalling the definition of the system of equations~\eqref{eq:FB_system1}, we have
\begin{empheq}[left=\empheqlbrace]{align}
X_{t_{i+1}}^{\pi} &= X_{t_i}^{\pi} + b(t_i,X_{t_i}^{\pi},Y_{t_i}^{\pi})h + \sigma(t_i,X_{t_i}^{\pi},Y_{t_i}^{\pi})\Delta W_i,  \\
Y_{t_{i+1}}^{\pi} &= Y_{t_i}^{\pi} - f(t_i,X_{t_i}^{\pi},Y_{t_i}^\pi,Z_{t_i}^\pi)h + (Z_{t_i}^\pi)\transpose\Delta W_i. \label{eq:discrete_Y}
\end{empheq}
Taking the expectation $E[\cdot|\mathcal{F}_{t_i}]$ on both sides of \eqref{eq:discrete_Y}, we obtain
\begin{equation*}
Y_{t_i}^{\pi} = E[Y_{t_{i+1}}^{\pi} + f(t_i,X_{t_i}^{\pi},Y_{t_i}^\pi,Z_{t_i}^\pi)h|\mathcal{F}_{t_i}].
\end{equation*}
Right multiplying $(\Delta W_i)\transpose$ on both sides of \eqref{eq:discrete_Y} and taking the expectation $E[\cdot|\mathcal{F}_{t_i}]$ again, we obtain
\begin{equation*}
Z_{t_i}^{\pi} = \frac{1}{h}[Y_{t_{i+1}}^{\pi}\Delta W_i|\mathcal{F}_{t_i}].
\end{equation*}

The above observation motivates us to consider the following system of equations
\begin{align}
\begin{dcases}
    X_0^{\pi} = \xi, \\
    X_{t_{i+1}}^{\pi} = X_{t_i}^{\pi} + b(t_i,X_{t_i}^{\pi},Y_{t_i}^{\pi})h + \sigma(t_i,X_{t_i}^{\pi},Y_{t_i}^{\pi})\Delta W_i, \\
    Z_{t_i}^{\pi} = \frac{1}{h}E[Y_{t_{i+1}}^{\pi}\Delta W_i|\mathcal{F}_{t_i}], \\
    Y_{t_i}^{\pi} = E[{Y}_{t_{i+1}}^{\pi} + f(t_i,{X}_{t_i}^{\pi},{Y}_{t_i}^{\pi},{Z}_{t_{i}}^{\pi})h|\mathcal{F}_{t_i}].
\end{dcases}
\label{eq:FB_system_analyzed}
\end{align}
Note that~\eqref{eq:FB_system_analyzed} is defined just like the FBSDEs \eqref{eq:FBSDE_forward}\eqref{eq:FBSDE_backward}, where the $X$ component is defined forwardly and the $Y,Z$ components are defined backwardly. However, since we do not specify the terminal condition of $Y_{T}^\pi$, the system of equations~\eqref{eq:FB_system_analyzed} has infinitely many solutions. The following lemma gives an estimate of the difference between two such solutions.

\begin{lem} \label{lem1}
For $j = 1, 2$, suppose $(\{X_{t_i}^{\pi,j}\}_{0\le i \le N}, \{Y_{t_i}^{\pi,j}\}_{0\le i \le N}, \{Z_{t_i}^{\pi,j}\}_{0\le i \le N-1})$ are two solutions of~\eqref{eq:FB_system_analyzed},
with $X_{t_i}^{\pi,j}, Y_{t_i}^{\pi,j} \in L^2(\Omega,\mathcal{F}_{t_i},\mathbb{P})$, $ 0 \le i \le N$. 
For any $\lambda_1 > 0, \lambda_2 \ge f_z$, and sufficiently small h, denote
\begin{equation}
\begin{aligned}
A_1 &\coloneqq 2k_b + \lambda_1+ \sigma_x + Kh, \\
A_2 &\coloneqq(\lambda_1^{-1}+h)b_y + \sigma_y, \\
A_3 &\coloneqq -\frac{\ln[1-(2k_f+\lambda_2)h]}{h}, \\
A_4 &\coloneqq \frac{f_x}{[1 - (2k_f + \lambda_2)h]\lambda_2}.
\end{aligned}
\label{eq:def_Ai}
\end{equation}
Let $\delta X_i = X_{t_i}^{\pi,1} - X_{t_i}^{\pi,2}, \delta Y_i = Y_{t_i}^{\pi,1} - Y_{t_i}^{\pi,2}$, then we have, for $ 0 \le n \le N$, 
\begin{align*}
E|\delta X_n|^2 &\le A_2\sum_{i = 0}^{n-1}e^{A_1(n-i-1)h}E|\delta Y_i|^2h, \\
E|\delta Y_n|^2 &\le e^{A_3(N-n)h}E|\delta Y_N|^2 + A_4\sum_{i = n}^{N-1}e^{A_3(i-n)h}E|\delta X_i|^2h.
\end{align*}
\end{lem}

To prove Lemma \ref{lem1}, we need the following lemma to handle the $Z$ component.
\begin{lem}\label{lem2}
Let $0 \le s_1 <  s_2$, given $Q \in L^2(\Omega,\mathcal{F}_{s_2},\mathbb{P})$, by the martingale representation theorem, there exists an $\mathcal{F}_t$-adapted process $\{H_s\}_{s_1 \le s \le s_2}$ such that $\int_{s_1}^{s_2} E|H_s|^2\,\mathrm{d} s < \infty$ and $Q = E[Q|\mathcal{F}_{s_1}] + \int_{s_1}^{s_2} H_s \,\mathrm{d} W_s$. 
Then we have $E[Q(W_{s_2} - W_{s_1})|\mathcal{F}_{s_1}] = E[\int_{s_1}^{s_2} H_s \, \mathrm{d}s|\mathcal{F}_{s_1}]$.
\end{lem}

\begin{proof}
Consider the auxiliary stochastic process $Q_s  = (E[Q|\mathcal{F}_{s_1}] + \int_{s_1}^s H_t \,\mathrm{d} W_t)(W_s - W_{s_1})$ for $s \in [s_1,s_2]$. By It\^{o}'s formula,
\begin{equation*}
\rmd Q_s = (W_s - W_{s_1})H_s \,\mathrm{d} W_s + (E[Q|\mathcal{F}_{s_1}] + \int_{s_1}^s H_t \,\mathrm{d} W_t) \,\mathrm{d} W_s + H_s \,\mathrm{d} s.
\end{equation*}
Noting that $Q_{s_1}=0$, we have 
\begin{equation*}
E[Q(W_{s_2} - W_{s_1})|\mathcal{F}_{s_1}] = E[Q_{s_2}|\mathcal{F}_{s_1}] = E[\int_{s_1}^{s_2}H_s \,\mathrm{d} s|\mathcal{F}_{s_1}].
\end{equation*}
\end{proof}

\begin{proof}[Proof of Lemma \ref{lem1}]
Let 
\begin{align*}
\delta Z_i &= Z_{t_i}^{\pi,1} - Z_{t_i}^{\pi,2},\\
\delta b_i &= b(t_i,X_{t_i}^{\pi,1},Y_{t_i}^{\pi,1}) - b(t_i,X_{t_i}^{\pi,2},Y_{t_i}^{\pi,2}), \\
\delta \sigma_i &= \sigma(t_i,X_{t_i}^{\pi,1},Y_{t_i}^{\pi,1}) - \sigma(t_i,X_{t_i}^{\pi,2},Y_{t_i}^{\pi,2}), \\
\delta f_i &= f(t_i,X_{t_i}^{\pi,1},Y_{t_i}^{\pi,1},Z_{t_i}^{\pi,1}) - f(t_i,X_{t_i}^{\pi,2},Y_{t_i}^{\pi,2},Z_{t_i}^{\pi,2}).
\end{align*}
Then we have
\begin{align}
\delta X_{i+1} &= \delta X_i + \delta b_i h + \delta \sigma_i \Delta W_i, \label{equa5} \\
\delta Z_i &= \frac{1}{h}E[\delta Y_{i+1}\Delta W_i|\mathcal{F}_{t_i}], \label{equa7}
 \\
\delta Y_i &= E[\delta Y_{i+1} + \delta f_i h|\mathcal{F}_{t_i}].
\end{align}
By the martingale representation theorem, there exists an $\mathscr{F}_t$-adapted square-integrable process $\{\delta Z_t\}_{t_i \le t \le t_{i+1}}$ such that $$\delta Y_{i+1} = E[\delta Y_{i+1}|\mathcal{F}_{t_i}] + \int_{t_i}^{t_{i+1}}(\delta Z_t)\transpose\, \mathrm{d}W_t,$$ or, equivalently,
\begin{align}
\delta Y_{i+1} &= \delta Y_{i} - \delta f_i h + \int_{t_i}^{t_{i+1}}(\delta Z_t)\transpose \, \mathrm{d}W_t. 
\label{equa6} 
\end{align}
From equations~\eqref{equa5} and~\eqref{equa6}, noting that $\delta X_i$, $\delta Y_i$, $\delta b_i$, $\delta \sigma_i$, and $\delta f_i$ are all $\mathcal{F}_{t_i}$-measurable, and $E[\Delta W_i|\mathcal{F}_{t_i}] = 0$, $E[\int_{t_i}^{t_{i+1}} (\delta Z_t)\transpose \, \mathrm{d}W_t|\mathcal{F}_{t_i}] = 0$, 
we have 
\begin{align}
E|\delta X_{i+1}|^2 &= E|\delta X_i + \delta b_i h|^2 + E[(\Delta W_i)\transpose(\delta \sigma_i)\transpose\delta \sigma_i \Delta W_i] \notag \\
&= E|\delta X_i + \delta b_i h|^2 + hE\|\delta \sigma_i\|^2, \label{equa8}\\
E|\delta Y_{i+1}|^2 &= E|\delta Y_i - \delta f_i h|^2 + \int_{t_i}^{t_{i+1}}E|\delta Z_t|^2 \, \mathrm{d}t. \label{equa9}
\end{align}

From equation \eqref{equa8}, by Assumptions \ref{assu1}, \ref{assu2} and the root-mean square and geometric mean inequality (RMS-GM inequality), for any $\lambda_1 > 0$, we have
\begin{align*}
& E|\delta X_{i+1}|^2 \notag \\
=~& E|\delta X_i|^2  + E|\delta b_i|^2 h^2 + hE\|\delta \sigma_i\|^2 \notag \\
& + 2hE[(b(t_i,X_{t_i}^{\pi,1},Y_{t_i}^{\pi,1}) -b(t_i,X_{t_i}^{\pi,2},Y_{t_i}^{\pi,1}))\transpose\delta X_i ] \notag \\
& + 2hE[(b(t_i,X_{t_i}^{\pi,2},Y_{t_i}^{\pi,1}) -b(t_i,X_{t_i}^{\pi,2},Y_{t_i}^{\pi,2}))\transpose\delta X_i ] \notag \\
\le~ & E|\delta X_i|^2 + (KE|\delta X_i|^2 + b_y E|\delta Y_i|^2)h^2 + 2k_bh E|\delta X_i|^2 \notag \\
& +(\lambda_1 E|\delta X_i|^2 + \lambda_1^{-1}b_y E|\delta Y_i|^2)h + (\sigma_x E|\delta X_i|^2 + \sigma_y E|\delta Y_i|^2)h\notag \\
=~& [1 + (2k_b + \lambda_1+ \sigma_x + Kh)h]E|\delta X_i|^2 + [(\lambda_1^{-1}+h)b_y + \sigma_y]E|\delta Y_i|^2h.
\end{align*}
Recall $A_1 = 2k_b + \lambda_1+ \sigma_x + Kh$, $A_2 =(\lambda_1^{-1}+h)b_y + \sigma_y$, $E|\delta X_0|^2 = 0$. By induction we can obtain that, for $ 0 \le n \le N$,
\begin{equation*}
E|\delta X_n|^2 \le A_2\sum_{i = 0}^{n-1}e^{A_1(n-i-1)h}E|\delta Y_i|^2h.
\end{equation*}

Similarly, from equation \eqref{equa9}, for any $\lambda_2 > 0$, we have
\begin{align}
&E|\delta Y_{i+1}|^2 \notag \\
\ge~& E|\delta Y_i|^2 +\int_{t_i}^{t_{i+1}}E|\delta Z_t|^2 \, \mathrm{d}t \notag \\
& - 2hE[(f(t_i,X_i^{1,\pi},Y_i^{1,\pi},Z_i^{1,\pi}) - f(t_i,X_i^{1,\pi},Y_i^{2,\pi},Z_i^{1,\pi}))\transpose\delta Y_i] \notag \\
& - 2hE[(f(t_i,X_i^{1,\pi},Y_i^{2,\pi},Z_i^{1,\pi}) - f(t_i,X_i^{2,\pi},Y_i^{2,\pi},Z_i^{2,\pi}))\transpose\delta Y_i] \notag \\
\ge~& E|\delta Y_i|^2 + \int_{t_i}^{t_{i+1}}E|\delta Z_t|^2 \, \mathrm{d}t - 2k_fh E|\delta Y_i|^2 \notag \\
&- [\lambda_2 E|\delta Y_i|^2 + \lambda_2^{-1}(f_xE|\delta X_i|^2 + f_zE|\delta Z_i|^2)]h.
\label{equa10}
\end{align}
To deal with the integral term in~\eqref{equa10}, we apply Lemma~\ref{lem2} to~\eqref{equa7}\eqref{equa6} and get
\begin{align*}
\delta Z_i &= \frac{1}{h}E[\int_{t_i}^{t_{i+1}}\delta Z_t \, \mathrm{d}t |\mathcal{F}_{t_i}],
\end{align*}
which implies, by the Cauchy inequality,
\begin{align*}
E|\delta Z_i|^2h &= \sum_{k=1}^d E|(\delta Z_i)_k|^2 h = \sum_{k=1}^{d}\frac{1}{h}E\Big|E[\int_{t_i}^{t_{i+1}}(\delta Z_t)_k \,\mathrm{d}t|\mathcal{F}_{t_i}]\Big|^2\notag \\& \le \sum_{k=1}^d\frac{1}{h} E\Big |\int_{t_i}^{t_{i+1}}(\delta Z_t)_k \,\mathrm{d}t \Big |^2 \le  \sum_{k=1}^d\int_{t_i}^{t_{i+1}} E|(\delta Z_t)_k|^2\,\rmd t \\& = \int_{t_i}^{t_{i+1}} E|\delta Z_t|^2\,\rmd t,
\end{align*}
where $(\cdot)_k$ denotes the $k$-th component of the vector.
Plugging it into~\eqref{equa10} gives us
\begin{align}
E|\delta Y_{i+1}|^2 \ge [1 - (2k_f + \lambda_2)h]E|\delta Y_i|^2 +(1-f_z\lambda_2^{-1})E|\delta Z_i|^2h - f_x\lambda_2^{-1} E|\delta X_i|^2h. 
\label{eq:lem1_delta_Y}
\end{align}

Then for any $\lambda_2 \ge f_z$ and sufficiently small $h$ satisfying $(2k_f+\lambda_2)h < 1$, we have
\begin{equation*}
E|\delta Y_i|^2 \le[1 - (2k_f + \lambda_2)h]^{-1}[E|\delta Y_{i+1}|^2 + f_x\lambda_2^{-1} E|\delta X_{i}|^2h].
\end{equation*}
Recall $A_3 = -h^{-1}\ln[1-(2k_f+\lambda_2)h]$, $A_4 = f_x\lambda_2^{-1}[1 - (2k_f + \lambda_2)h]^{-1} $. By induction we obtain that, for $ 0 \le n \le N$, 
\begin{equation*}
E|\delta Y_n|^2 \le e^{A_3(N-n)h}E|\delta Y_N|^2 + A_4\sum_{i = n}^{N-1}e^{A_3(i-n)h}E|\delta X_i|^2h.
\end{equation*}
\end{proof}

Now we are ready to prove Theorem \ref{thm:main1}, whose precise statement is given below. Note that its conditions are satisfied if any of the five cases in the weak coupling and monotonicity conditions holds.
\begin{thmbis}{thm:main1}\label{thm:main1_detail}
Suppose Assumptions \ref{assu1}, \ref{assu2}, \ref{assu3}, and \ref{assu4} hold true and there exist $\lambda_1 > 0, \lambda_2 \ge f_z$ such that $\overline{A_0}<1$, where
\begin{equation}
\begin{aligned}
& \overline{A_1} \coloneqq 2k_b + \lambda_1 + \sigma_x, \\
& \overline{A_2} \coloneqq b_y\lambda_1^{-1} + \sigma_y, \\
& \overline{A_3} \coloneqq 2k_f+\lambda_2, \\
& \overline{A_4} \coloneqq f_x\lambda_2^{-1}, \\
& \overline{A_0} \coloneqq \overline{A_2}\frac{1-e^{-(\overline{A_1}+\overline{A_3})T}}{\overline{A_1}+\overline{A_3}} \Big\{g_x e^{(\overline{A_1}+\overline{A_3})T} + \overline{A_4}\frac{e^{(\overline{A_1}+\overline{A_3})T}-1}{\overline{A_1}+\overline{A_3}} \Big\}. 
\end{aligned}
\end{equation}

Then there exists a constant $C > 0$, depending on $E|\xi|^2$, $ \mathscr{L}$, $T$, $\lambda_1$, and $\lambda_2$, such that for sufficiently small $h$,
\begin{equation}
    \sup_{ t \in [0,T]} (E|X_t - \hat{X}_t^{\pi}|^2 + E|Y_t - \hat{Y}_t^{\pi}|^2) + \int_{0 }^{T} E|Z_t - \hat{Z}_t^\pi|^2 \, \mathrm{d}t \le C[h +  E|g(X_{T}^\pi) - Y_T^{\pi}|^2],
\end{equation}
where $\hat{X}_t^\pi = X_{t_i}^\pi$, $\hat{Y}_t^\pi = Y_{t_i}^\pi$, $\hat{Z}_t^\pi = Z_{t_i}^{\pi}$ for $t \in [t_i,t_{i+1})$.
\end{thmbis}
\begin{rem}
The above theorem also implies the coercivity of the objective function~\eqref{eq:disc_objective} used in the deep BSDE method. Formally speaking, the coercivity means that if $\sum_{i = 0}^{N-1}E|Z_{t_i}^\pi|^2 + E|Y_0^\pi|^2 \rightarrow +\infty$, we have $E|g(X_T^\pi) - Y_T^{\pi}|^2 \rightarrow + \infty$, which is a direct result from Theorem \ref{thm:main1_detail}.
\end{rem}
\medskip

\begin{rem}
\label{rem:lambda_existence}
If any of the weak coupling and monotonicity conditions introduced in Assumption \ref{assu3} holds to a sufficient extent, there must exist $\lambda_1, \lambda_2$ satisfying the conditions in Theorem \ref{thm:main1_detail}. We discuss the 5 cases in what follows.
\begin{enumerate}[\indent 1.]
\item Suppose all other constants and $\lambda_1 >0, \lambda_2\geq f_z$ are fixed, if $T>0$ is sufficiently small, then the second factor of $\overline{A_0}$ could be sufficiently close to 0 such that $\overline{A_0} < 1$.
\item Suppose all other constants and $\lambda_1 >0, \lambda_2\geq f_z$ are fixed, if $b_y\geq 0$ and $\sigma_y \geq 0$ are sufficiently small, then $\overline{A_2} \geq 0$ could be sufficiently small such that $\overline{A_0} < 1$. 
\item Suppose all other constants and $\lambda_1 >0, \lambda_2\geq f_z$ are fixed, if $f_x \geq 0$ and $g_x \geq 0$ are sufficiently small, then $\overline{A_4}$ and thus the last factor in $\overline{A_0}$ could be sufficiently close to 0 such that $\overline{A_0}<1$.
\item Suppose all constants except $k_f$ and $\lambda_2>0$ are fixed.
Let $\overline{A_1}' \coloneqq \overline{A_1} + \overline{A_3} = 
2k_b + 2k_f + \sigma_x + \lambda_1 + \lambda_2$ and rewrite $\overline{A_0}$ as
\begin{align*}
\overline{A_0} = \overline{A_2}\Big\{g_x \frac{e^{\overline{A_1}'T} - 1}{\overline{A_1}'} + \overline{A_4}\frac{e^{\overline{A_1}'T}+e^{-\overline{A_1}'T}-2}{(\overline{A_1}')^2} \Big\}.
\end{align*}
It is straightforward to check that there exists a negative constant $C_1$ such that when $\overline{A_1}' \leq C_1$, $(e^{\overline{A_1}'T}-1)/\overline{A_1}' < 1/(2\overline{A_2}g_x)$.
By the definition of $\overline{A_1}'$, if $k_f$ is sufficiently negative, there exists $\lambda_2 \geq f_x$ such that $\overline{A_1}'=C_1$ and $\lambda_2$ is sufficiently large to ensure
\begin{align*}
\overline{A_2}\,\overline{A_4}\frac{e^{C_1T} + e^{-C_1T}-2}{C_1^2}=\frac{f_x\overline{A_2}(e^{C_1T} + e^{-C_1T}-2)}{\lambda_2C_1^2} < \frac12.
\end{align*}
Combining these two estimates gives $\overline{A_0}<1$.
\item Noting that $k_b$ and $k_f$ play the same role in $\overline{A_1}'$, we use the same argument as above to show that when $k_b$ is sufficiently negative, there exists $\lambda_2 \geq f_x$ such that $\overline{A_0}<1$.
\end{enumerate}
\end{rem}
\medskip

\begin{proof}[Proof of Theorem \ref{thm:main1_detail}]
From the proof of this theorem and throughout the remainder of the paper, we use $C$ to generally denote a constant that only depends on $E|\xi|^2$, $ \mathscr{L}$, and $T$, whose value may change from line to line when there is no need to distinguish. We also use $C(\cdot)$ to generally denote a constant depending on $E|\xi|^2$, $ \mathscr{L}$, $T$ and the constants represented by $\cdot$. 

We use the same notations as Lemma \ref{lem1}. Let $X_{t_i}^{\pi,1} = X_{t_i}^\pi$, $Y_{t_i}^{\pi,1} = Y_{t_i}^{\pi}, Z_{t_i}^{\pi,1} = Z_{t_i}^\pi$ (defined in system~\eqref{eq:FB_system1}) and $X_{t_i}^{\pi,2} = \overline{X}_{t_i}^{\pi}$, $Y_{t_i}^{\pi,2} = \overline{Y}_{t_i}^{\pi}$, $Z_{t_i}^{\pi,2} = \overline{Z}_{t_i}^{\pi}$ (defined in system~\eqref{eq:discrete_FBSDE}). It can be easily checked that both $(\{X_{t_i}^{\pi,j}\}_{0\le i \le N}$, $\{Y_{t_i}^{\pi,j}\}_{0\le i \le N}$, $\{Z_{t_i}^{\pi,j}\}_{0\le i \le N-1})$, $j=1,2$ satisfy the system of equations~\eqref{eq:FB_system_analyzed}. Our proof strategy is to use Lemma~\ref{lem1} to bound %
the difference between two solutions
through the objective function $E|g(X_T^{\pi}) - Y_T^{\pi}|^2$. This allows us to apply Theorem \ref{thm5} to derive the desired estimates.

To begin with, note that for any $\lambda_3 >0$, the RMS-GM inequality yields
\begin{equation*}
E|\delta Y_N|^2 = E|g(\overline{X}_T^\pi) - Y_T^\pi|^2 \le (1 + \lambda_3^{-1})E|g(X_T^{\pi}) - Y_T^{\pi}|^2 + g_x(1 + \lambda_3)E|\delta X_N|^2.
\end{equation*}
Let 
\begin{align*}
P = \displaystyle{\max_{0\le n \le N}e^{-A_1nh}E|\delta X_n|^2}, \quad
S = \displaystyle{\max_{0\le n \le N} e^{A_3nh}E|\delta Y_n|^2}.
\end{align*}
Lemma \ref{lem1} tells us
\begin{equation*}
e^{-A_1nh}E|\delta X_n|^2 \le A_2\sum_{i=0}^{n-1}e^{-A_1(i+1)h}E|\delta Y_i|^2h \le A_2S\sum_{i=0}^{n-1}e^{-A_1(i+1)h-A_3ih}h,
\end{equation*}
and
\begin{align*}
&e^{A_3nh}E|\delta Y_n|^2 \notag \\
\le & e^{A_3T}E|\delta Y_N|^2 +  A_4\sum_{i = n}^{N-1}e^{A_3ih}E|\delta X_i|^2h  \\
\le &e^{A_3T}[(1 + \lambda_3^{-1})E|g(X_T^{\pi}) - Y_T^{\pi}|^2 + g_x(1 + \lambda_3)E|\delta X_N|^2] + A_4\sum_{i = n}^{N-1}e^{A_3ih}E|\delta X_i|^2h  \\
\le &e^{A_3T}(1 + \lambda_3^{-1})E|g(X_T^{\pi}) - Y_T^{\pi}|^2 + [g_x(1+\lambda_3)e^{(A_1+A_3)T} + A_4\sum_{i=n}^{N-1}e^{(A_1 + A_3)ih}h]P.
\end{align*}
Therefore by definition of $P$ and  $S$, we have 
\begin{align*}
P &\le A_2he^{-A_1h}\frac{e^{-(A_1+A_3)T} - 1}{e^{-(A_1+A_3)h}-1}S, \\
S &\le e^{A_3T}(1+\lambda_3^{-1})E|g(X_T^\pi) - Y_T^\pi|^2 + [g_x(1+\lambda_3)e^{(A_1+A_3)T}+ A_4h\frac{e^{(A_1+A_3)T}-1}{e^{(A_1+A_3)h }-1}]P.
\end{align*}
Consider the function
\begin{align*}
A(h) & = A_2he^{-A_1h}\frac{e^{-(A_1+A_3)T} - 1}{e^{-(A_1+A_3)h}-1}[g_x(1+\lambda_3)e^{(A_1+A_3)T}+ A_4h\frac{e^{(A_1+A_3)T}-1}{e^{(A_1+A_3)h }-1}].
\end{align*}
When $A(h) < 1$, we have
\begin{align*}
P &\le [1 - A(h)]^{-1}e^{A_3T}(1+\lambda_3^{-1})A_2h e^{-A_1h}\frac{e^{-(A_1+A_3)T} - 1}{e^{-(A_1+A_3)h}-1}E|g(X_T^{\pi}) - Y_T^{\pi}|^2, \\
S &\le [1 - A(h)]^{-1}e^{A_3T}(1+\lambda_3^{-1})E|g(X_T^{\pi}) - Y_T^{\pi}|^2.
\end{align*}

Let
\begin{equation}
\overline{P} = \max_{0\le n\le N}e^{-\overline{A_1}nh}E|\delta X_n|^2, \quad \overline{S} = \max_{0\le n \le N}e^{\overline{A_3}nh}E|\delta Y_n|^2. 
\end{equation}
Recall
\begin{align*}
\lim_{h\rightarrow 0}A_i = \overline{A_i}, \quad i=1,2,3,4,
\end{align*}
and note that
\begin{align*}
\lim_{h\rightarrow 0}A(h) = \overline{A_2}\frac{1-e^{-(\overline{A_1}+\overline{A_3})T}}{\overline{A_1}+\overline{A_3}}[g_x(1+\lambda_3)e^{(\overline{A_1}+\overline{A_3})T} + \overline{A_4}\frac{e^{(\overline{A_1}+\overline{A_3})T}-1}{\overline{A_1}+\overline{A_3}}].
\end{align*}
When $\overline{A_0} < 1$, comparing $\lim_{h\rightarrow 0}A(h)$ and $\overline{A_0}$, we know that, for any $\epsilon > 0$, there exists $\lambda_3 > 0$ and sufficiently small $h$ such that
\begin{align}
\overline{P} &\le (1+\epsilon)[1 - \overline{A_0}]^{-1} \overline{A_2} e^{\overline{A_3}T}(1+\lambda_3^{-1})\frac{ 1- e^{-(\overline{A_1}+\overline{A_3})T} }{\overline{A_1}+\overline{A_3}}E|g(X_T^{\pi}) - Y_T^{\pi}|^2, \label{eq:overline_P} \\
\overline{S} &\le (1+\epsilon)[1 - \overline{A_0}]^{-1}e^{\overline{A_3}T}(1+\lambda_3^{-1})E|g(X_T^{\pi}) - Y_T^{\pi}|^2.
\end{align}
By fixing $\epsilon = 1$ and choosing suitable $\lambda_3$, we obtain our error estimates of $E|\delta X_n|^2$ and $E|\delta Y_n|^2$ as
\begin{align}
\max_{0\le n \le N}E|\delta X_n|^2 &\le e^{\overline{A_1}T\vee 0} \overline{P}
\le C(\lambda_1,\lambda_2) E|g(X_T^{\pi}) - Y_T^{\pi}|^2, \label{equa11} \\
\max_{0\le n \le N}E|\delta Y_n|^2 &\le e^{(-\overline{A_3}T)\vee 0}\overline{S}
\le C(\lambda_1,\lambda_2) E|g(X_T^{\pi}) - Y_T^{\pi}|^2. \label{equa12}
\end{align}

To estimate $E|\delta Z_n|^2$, we consider estimate~\eqref{eq:lem1_delta_Y}, in which $\lambda_2$ can take any value no smaller than $f_z$. If $f_z \neq 0$, we choose $\lambda_2 = 2f_z$ 
and obtain
\begin{equation*}
\frac{1}{2}E|\delta Z_i|^2 h \le \frac{f_x}{2f_z}E|\delta X_i|^2h + E|\delta Y_{i+1}|^2 - [1-(2k_f+2f_z)h]E|\delta Y_i|^2.
\end{equation*}
Summing from 0 to $N-1$ gives us
\begin{align}
\sum_{i = 0}^{N-1}E|\delta Z_i|^2h 
\le~ & \frac{f_xT}{f_z}\max_{0\le n\le N}E|\delta X_n|^2 + [4(k_f+f_z)T \vee 0 + 2]\max_{0\le n\le N}E|\delta Y_n|^2 \notag \\
\le~ & C(\lambda_1,\lambda_2) E|g(X_T^{\pi}) - Y_T^{\pi}|^2.  \label{equa14}
\end{align}
The case $f_z=0$ can be dealt with similarly by choosing $\lambda_2 = 1$ and the same type of estimate can be derived.
Finally, combining estimates \eqref{equa11}\eqref{equa12}\eqref{equa14} with Theorem~\ref{thm5}, we prove the statement in Theorem~\ref{thm:main1_detail}.
\end{proof}

\section{An Upper Bound for the Minimized Objective Function}
We prove Theorem \ref{thm:main2} in this section. We first state three useful lemmas. Theorem~\ref{thm:main2_detail}, as a detailed statement of Theorem~\ref{thm:main2}, and Theorem~\ref{thm:main2'}, as an variation of Theorem~\ref{thm:main2_detail} under stronger conditions, are then provided, followed by their proofs. The proofs of three lemmas are given at the end of the section.

The main process we analyze is~\eqref{eq:FB_system1}. Lemma~\ref{lem3} gives an estimate of the final distance $E|g(X_T^\pi) - Y_T^{\pi}|^2$ provided by~\eqref{eq:FB_system1} in terms of the deviation between the approximated variables $Y_0^\pi, Z_{t_i}^\pi$ and the true solutions.

\begin{lem} \label{lem3}
Suppose Assumptions \ref{assu1}, \ref{assu2}, and \ref{assu3} hold true. Let $X_T^\pi, Y_0^{\pi}, Y_T^{\pi}$, $\{Z_{t_i}^{\pi}\}_{0\le i \le N-1}$ be defined as in system~\eqref{eq:FB_system1} and $\tilde{Z}_{t_i} =h^{-1}E[\int_{t_i}^{t_{i+1}}Z_t \,\rmd t|\mathcal{F}_{t_i}]$. Given $ \lambda_4 > 0$, 
there exists a constant $C > 0$ depending on $E|\xi|^2$, $\mathscr{L}$, $T$, and $\lambda_4$, such that for sufficiently small $h$,
\begin{align*}
 E|g(X_T^\pi) - Y_T^{\pi}|^2 
\le  (1 + \lambda_4)H_{\mathrm{min}}\sum_{i=0}^{N-1}E|\delta \tilde{Z}_{t_i}|^2h + C[h + E|Y_0 - Y_0^{\pi}|^2],
\end{align*}
where $\delta \tilde{Z}_{t_i} = \tilde{Z}_{t_i} - Z_{t_i}^\pi$, 
$ H(x) =(1 + \sqrt{g_x})^2e^{(2K + 2Kx^{-1} + x)T}(1 + f_zx^{-1})$, and $H_{\mathrm{min}} = \min_{x \in R^{+}} H(x)$.
\end{lem}

Lemma~\ref{lem3} is close to Theorem~\ref{thm:main2}, except that $\tilde{Z}_{t_i}$ is not a function of $X_{t_i}^\pi$ and $Y_{t_i}^\pi$ defined in \eqref{eq:FB_system1}. To bridge this gap, we need the following two lemmas. 
First, similar to the proof of Theorem \ref{thm:main1_detail}, an estimate of the distance between the process defined in~\eqref{eq:FB_system1} and the process defined in~\eqref{eq:discrete_FBSDE} is also needed here. Lemma~\ref{lem4} is a general result to serve this purpose, providing an estimate of the difference between two backward processes driven by different forward processes.

\begin{lem} \label{lem4}
Let $X_{t_i}^{\pi,j} \in L^2(\Omega,\mathcal{F}_{t_i},\mathbb{P})$ for $0 \le i \le N$, $ j = 1, 2$. Suppose $\{Y_{t_i}^{\pi,j}\}_{0\le i \le N}$ and $\{Z_{t_i}^{\pi,j}\}_{0\le i \le N-1}$ satisfy
\begin{align}
\begin{dcases}
Y_{T}^{\pi,j} = g(X_{T}^{\pi,j}), \\
Z_{t_i}^{\pi,j} = \frac{1}{h}E[Y_{t_{i+1}}^{\pi,j}\Delta W_i|\mathcal{F}_{t_i}],  \\
Y_{t_i}^{\pi,j} = E[{Y}_{t_{i+1}}^{\pi,j} + f(t_i,{X}_{t_i}^{\pi,j},{Y}_{t_i}^{\pi,j},{Z}_{t_{i}}^{\pi,j})h|\mathcal{F}_{t_i}],
\label{eq:lem4_def}
\end{dcases}
\end{align}
for $0 \le i \le N-1$, $j = 1, 2$. Let $\delta X_i = X_{t_i}^{\pi,1} - X_{t_i}^{\pi,2}$, $\delta Z_i = Z_{t_i}^{\pi,1} - Z_{t_i}^{\pi,2}$, then for any $\lambda_7 > f_z$, and sufficiently small $h$, we have
\begin{align*}
\sum_{i=0}^{N-1}E|\delta Z_i|^2h \le \frac{\lambda_7(e^{-A_5T}\vee 1)}{\lambda_7 - f_z}\Big\{g_x e^{A_5T - A_5h}E|\delta X_N|^2 + \frac{f_x}{\lambda_7}\sum_{i=0}^{N-1}e^{A_5ih}E|\delta X_i|^2h\Big\},
\end{align*}
where $A_5 \coloneqq -h^{-1}\ln[1 -(2k_f+\lambda_7)h]$.
\end{lem}

Lemma~\ref{lem5} shows that, similar to the nonlinear Feynman--Kac formula, the discrete stochastic process defined in~\eqref{eq:FB_system1} can also be linked to some deterministic functions.

\begin{lem}\label{lem5}
Let $\{X_{t_i}^\pi\}_{0\le i\le N}, \{Y_{t_i}^\pi\}_{0\le i\le N}$ be defined in \eqref{eq:FB_system1}. When $h < 1/\sqrt{K}$, there exist deterministic functions $U_i^{\pi}: \bR^m \times \bR \rightarrow \bR, V_i^{\pi}: \bR^m \times \bR \rightarrow \bR^d$ for $0 \le i \le N$ such that $Y_{t_i}^{\pi, '}=U_i^\pi(X_{t_i}^\pi, Y_{t_i}^\pi)$, $Z_{t_i}^{\pi, '}=V_i^\pi(X_{t_i}^\pi, Y_{t_i}^\pi)$ satisfy
\begin{align}
\begin{dcases}
Y_{t_N}^{\pi,'} = g(X_{t_N}^{\pi}), \\
Z_{t_i}^{\pi,'} = \frac{1}{h}E[Y_{t_{i+1}}^{\pi,'}\Delta W_i|\mathcal{F}_{t_i}],\\
Y_{t_i}^{\pi,'} = E[{Y}_{t_{i+1}}^{\pi,'} + f(t_i,{X}_{t_i}^{\pi},{Y}_{t_i}^{\pi,'},{Z}_{t_{i}}^{\pi,'})h|\mathcal{F}_{t_i}],
\end{dcases}
\label{eq:lem5_system}
\end{align}
for  $0 \le i \le N-1$.
If $b$ and $\sigma$ are independent of $y$, then there exist deterministic functions $U_i^{\pi}: \bR^m \rightarrow \bR, V_i^{\pi}: \bR^m \rightarrow \bR^d$ for $0 \le i \le N$ such that $Y_{t_i}^{\pi, '}=U_i^\pi(X_{t_i}^\pi)$, $Z_{t_i}^{\pi, '}=V_i^\pi(X_{t_i}^\pi)$ satisfy \eqref{eq:lem5_system}.
\end{lem}

Now we are ready to prove Theorem \ref{thm:main2}, with a precise statement given below. Like Theorem~\ref{thm:main1_detail}, the conditions below are satisfied if any of the five cases of the weak coupling and monotonicity conditions holds to certain extent.
\begin{thmbis}{thm:main2}\label{thm:main2_detail}
Suppose Assumptions \ref{assu1}, \ref{assu2},  \ref{assu3}, and \ref{assu4} hold true. Given any $\lambda_1,\lambda_3 > 0$, $\lambda_2 \ge f_z$, and $\lambda_7 > f_z$,  let $\overline{A_i}, (i=1,2,3,4)$ be defined in \eqref{eq:def_Ai} and
\begin{equation}
\begin{aligned}
\overline{A_5} \coloneqq &~ \lambda_7 + 2k_f, \\
\overline{A_0}' \coloneqq &~ \overline{A_2}\frac{1-e^{-(\overline{A_1}+\overline{A_3})T}}{\overline{A_1}+\overline{A_3}} \Big\{g_x(1+\lambda_3)e^{(\overline{A_1}+\overline{A_3})T} + \overline{A_4}\frac{e^{(\overline{A_1}+\overline{A_3})T}-1}{\overline{A_1}+\overline{A_3}} \Big\}, \\
\overline{B_0} \coloneqq &~ H_{\mathrm{min}}\overline{A_2}e^{\overline{A_3}T} \frac{ 1- e^{-(\overline{A_1}+\overline{A_3})T} }{\overline{A_1}+\overline{A_3}}[1 - \overline{A_0}']^{-1}(1+\lambda_3^{-1})
\\
&~\times \frac{\lambda_7(e^{-\overline{A_5}T}\vee 1)}{\lambda_7 - f_z}\Big\{g_xe^{(\overline{A_1}+\overline{A_5})T} + \frac{f_x}{\lambda_7}\frac{e^{(\overline{A_1}+\overline{A_5})T}-1}{\overline{A_1}+\overline{A_5}} \Big\}.
\end{aligned}
\end{equation}
If there exist $\lambda_1, \lambda_2, \lambda_3, \lambda_7$ satisfying $\overline{A_0}' < 1$ and $\overline{B_0} < 1$, then there exists a constant C depending on $E|\xi|^2$, $\mathscr{L}$, $T$, $\lambda_1$, $\lambda_2$, $\lambda_3$, and $\lambda_7$, such that for sufficiently small $h$,
\begin{equation}
E|g(X_T^\pi) - Y_T^\pi|^2 \le C\Big\{h + E|Y_0 - Y_0^\pi|^2 + \sum_{i=0}^{N-1}E|E[\tilde{Z}_{t_i}|X_{t_i}^\pi,Y_{t_i}^\pi] - Z_{t_i}^\pi|^2h \Big\},
\label{eq:thm2_detail_res}
\end{equation}
where $\tilde{Z}_{t_i} = h^{-1}E[\int_{t_i}^{t_{i+1}}Z_t \,\mathrm{d} t|\mathcal{F}_{t_i}]$. If $Z_t$ is c\'adlag, we can replace $\tilde{Z}_{t_i}$ with $Z_{t_i}$. If $b$ and $\sigma$ are independent of $y$, we can replace $E[\tilde{Z}_{t_i}|X_{t_i}^\pi,Y_{t_i}^\pi]$ with $E[\tilde{Z}_{t_i}|X_{t_i}^\pi]$.
\end{thmbis}

\begin{rem}
If we take the infimum within the domains of $Y_0^\pi$ and $Z_{t_i}^\pi$ on both sides, we recover the original statement in Theorem \ref{thm:main2}.
\end{rem}
\medskip

\begin{rem}
If any of the weak coupling and monotonicity conditions introduced in Assumption \ref{assu3} holds to a sufficient extent, there must exist $\lambda_1, \lambda_2,  \lambda_3,  \lambda_7$ satisfying the conditions in Theorem \ref{thm:main2_detail}. The arguments are very similar to those provided in Remark \ref{rem:lambda_existence}. Hence, we omit the details here for the sake of brevity.
\end{rem}

\begin{proof}[Proof of Theorem \ref{thm:main2_detail}]
Using Lemma \ref{lem3} with $\lambda_4>0$, we obtain
\begin{align}
E|g(X_T^\pi) - Y_T^{\pi}|^2 \le(1 + \lambda_4) H_{\mathrm{min}} 
\sum_{i=0}^{N-1}E|\delta \tilde{Z}_{t_i}|^2h + C(\lambda_4)[h+E|Y_0 - Y_0^{\pi}|^2].
\label{eq:thm2_1}
\end{align}
Splitting the term $\delta \tilde{Z}_{t_i} = \tilde{Z}_{t_i} - Z_{t_i}^\pi$ and applying the generalized mean inequality, we have (recall $\overline{Z}_{t_i}^\pi$ is defined in Theorem~\ref{thm5})
\begin{align}
&~ E|\delta \tilde{Z}_{t_i}|^2 \notag \\
\le &~(1 + \lambda_4)E|\overline{Z}_{t_i}^\pi - E[\overline{Z}_{t_i}^\pi|X_{t_i}^\pi,Y_{t_i}^\pi]|^2  \notag \\
&~ + (1+\lambda_4^{-1})\Big\{E|(\tilde{Z}_{t_i} -\overline{Z}_{t_i}^\pi) - E[(\tilde{Z}_{t_i} -\overline{Z}_{t_i}^\pi)|X_{t_i}^\pi,Y_{t_i}^\pi] \notag \\
&~ \qquad\qquad\qquad
+ (E[\tilde{Z}_{t_i}|X_{t_i}^\pi,Y_{t_i}^\pi] - Z_{t_i}^\pi)|^2\Big\} \notag \\
\le &~(1 + \lambda_4)E|\overline{Z}_{t_i}^\pi - E[\overline{Z}_{t_i}^\pi|X_{t_i}^\pi,Y_{t_i}^\pi]|^2 
 \notag \\
&~ + 3(1+\lambda_4^{-1})\Big\{E|\tilde{Z}_{t_i} - \overline{Z}_{t_i}^\pi|^2 + E|E[(\tilde{Z}_{t_i} - \overline{Z}_{t_i}^\pi)|X_{t_i}^\pi,Y_{t_i}^\pi]|^2  \notag \\
&~ \qquad\qquad\qquad~
+ E|E[\tilde{Z}_{t_i}|X_{t_i}^\pi,Y_{t_i}^\pi] - Z_{t_i}^\pi|^2\Big\} \notag \\
\le &~(1 + \lambda_4)E|\overline{Z}_{t_i}^\pi - E[\overline{Z}_{t_i}^\pi|X_{t_i}^\pi,Y_{t_i}^\pi]|^2 
 \notag \\
&~ + 3(1+\lambda_4^{-1})\Big\{2E|\tilde{Z}_{t_i} - \overline{Z}_{t_i}^\pi|^2
+ E|E[\tilde{Z}_{t_i}|X_{t_i}^\pi,Y_{t_i}^\pi] - Z_{t_i}^\pi|^2\Big\}. \label{eq:thm2_2}
\end{align}
From equations~\eqref{eq:path_regularity}\eqref{eq:discrete_convergece}, we know that
\begin{align}
\sum_{i=0}^{N-1}E|\tilde{Z}_{t_i} - \overline{Z}_{t_i}^\pi|^2h &\leq 2\sum_{i=0}^{N-1}\int_{t_i}^{t_{i+1}}[E|Z_t - \tilde{Z}_{t_i}|^2 + E|Z_t - \overline{Z}_{t_i}^\pi|^2  \rmd t  \notag \\
&= 2\int_{0}^{T}[E|Z_t - \tilde{Z}_{t_i}|^2 + E|Z_t - \overline{Z}_{t_i}^\pi|^2  \rmd t \notag \\
&\leq  C(1+E|\xi|^2)h.
\label{eq:thm2_3}
\end{align}

Plugging estimates \eqref{eq:thm2_2}\eqref{eq:thm2_3} into \eqref{eq:thm2_1} gives us
\begin{align}
&~E|g(X_T^\pi) - Y_T^{\pi}|^2 \notag \\
\le &~ (1 + \lambda_4)^2H_{\mathrm{min}} \sum_{i=0}^{N-1}
E|\overline{Z}_{t_i}^\pi - E[\overline{Z}_{t_i}^\pi|X_{t_i}^\pi,Y_{t_i}^\pi]|^2 h  \notag \\
&~ + C(\lambda_4)\Big\{h + E|Y_0 - Y_0^{\pi}|^2 + \sum_{i=0}^{N-1}
E|E[\tilde{Z}_{t_i}|X_{t_i}^\pi,Y_{t_i}^\pi] - Z_{t_i}^\pi|^2h \Big\}. \label{eq:thm2_4}
\end{align}

It remains to estimate the term $ \sum_{i=0}^{N-1}
E|\overline{Z}_{t_i}^\pi - E[\overline{Z}_{t_i}^\pi|X_{t_i}^\pi,Y_{t_i}^\pi]|^2 h$, to which we intend to apply Lemma~\ref{lem4}.
Let $X_{t_i}^{\pi,1} = X_{t_i}^{\pi}$ and $X_{t_i}^{\pi,2} = \overline{X}_{t_i}^{\pi}$. The associated $Z_{t_i}^{\pi,1}$ and $Z_{t_i}^{\pi,2}$ are then defined according to equation~\eqref{eq:lem4_def}. Note that $Z_{t_i}^{\pi,2} = \overline{Z}_{t_i}^{\pi}$ but $Z_{t_i}^{\pi,1}$ is not necessarily equal to $Z_{t_i}^{\pi}$, due to the possible violation of the terminal condition. 
From Lemma \ref{lem5}, we know $Z_{t_i}^{\pi,1}$ can be represented as $V_i^\pi(X_{t_i}^\pi,Y_{t_i}^\pi)$ with $V_i^\pi$ being a deterministic function. By the property of conditional expectation, we have
\begin{align*}
E|\overline{Z}_{t_i}^\pi - E[\overline{Z}_{t_i}^\pi|X_{t_i}^\pi,Y_{t_i}^{\pi}]|^2 \leq
E|\overline{Z}_{t_i}^\pi - V_i(X_{t_i}^\pi,Y_{t_i}^{\pi})|^2,
\end{align*}
for any $V_i$. Therefore we have the estimate
\begin{align}
&~\sum_{i=0}^{N-1}E|\overline{Z}_{t_i}^\pi -  E[\overline{Z}_{t_i}^\pi|X_{t_i}^\pi,Y_{t_i}^{\pi}]|^2h 
\le \sum_{i=0}^{N-1}E|\delta Z_i|^2h  \notag \\
\le &~ \frac{\lambda_7(e^{-A_5T}\vee 1)}{\lambda_7 - f_z} \Big\{g_xe^{A_5T -A_5h}E|\delta X_N|^2+ \frac{f_x}{\lambda_7}\sum_{i=0}^{N-1}e^{A_5ih}E|\delta X_i|^2h \Big\}. \label{eq:sum_delta_Z}
\end{align}
Recall that $\delta X_i=X_{t_i}^\pi - \overline{X}_{t_i}^\pi, \delta Z_i=Z_{t_i}^{\pi,1} - \overline{Z}_{t_i}^\pi$.
Similar to the derivation of estimate~\eqref{eq:overline_P} (using a given $\lambda_3>0$ without final specification) in the proof of Theorem~\ref{thm:main1_detail}, when $\overline{A_0}' < 1$, we have 
\begin{align}
\overline{P} &\le (1+\lambda_4)\overline{A_2}e^{\overline{A_3}T}\frac{ 1- e^{-(\overline{A_1}+\overline{A_3})T} }{\overline{A_1}+\overline{A_3}}[1 - \overline{A_0}']^{-1}(1+\lambda_3^{-1})E|Y_T^{\pi}-g(X_T^{\pi})|^2,
\label{eq:overline_P_prime}
\end{align}
in which $\overline{P} = \max_{0\le n \le N}e^{-\overline{A_1}nh}E|\delta X_i|^2$.
Plugging~\eqref{eq:overline_P_prime} into~\eqref{eq:sum_delta_Z}, and then into~\eqref{eq:thm2_4}, we get
\begin{align*}
\sum_{i=0}^{N-1}E|\delta Z_i|^2h&\le \frac{\lambda_7(e^{-A_5T}\vee1)\overline{P}}{\lambda_7-f_z}\Big\{g_xe^{(\overline{A_1}+A_5)T - A_5h} + \frac{f_x}{\lambda_7}\sum_{i=0}^{N-1}e^{(\overline{A_1}+A_5)ih}h\Big\},
\end{align*}
and
\begin{align}
&~E|g(X_T^\pi) - Y_T^{\pi}|^2  \notag \\
\le &~ (1 + \lambda_4)^3 B(h)E|g(X_T^\pi) - Y_T^{\pi}|^2 \notag  \\
&~ + C(\lambda_4)\Big\{h + E|Y_0 - Y_0^\pi|^2 + \sum_{i=0}^{N-1}E|E[\tilde{Z}_{t_i}|X_{t_i}^\pi,Y_{t_i}^\pi] - Z_{t_i}^\pi|^2h \Big\}, \label{eq:thm2_5}
\end{align}
for sufficiently small $h$. Here $B(h)$ is defined as
\begin{align*}
B(h) =&~ H_{\mathrm{min}}\overline{A_2}e^{\overline{A_3}T} \frac{ 1- e^{-(\overline{A_1}+\overline{A_3})T} }{\overline{A_1}+\overline{A_3}}[1 - \overline{A_0}']^{-1}(1+\lambda_3^{-1})
\notag \\
&~\times \frac{\lambda_7(e^{-A_5T}\vee 1)}{\lambda_7 - f_z}\Big\{g_xe^{(\overline{A_1}+A_5)T-A_5h} + \frac{f_x}{\lambda_7}\sum_{i=0}^{N-1}e^{(\overline{A_1}+A_5)ih}h \Big\}.
\end{align*}
The forms of inequalities~\eqref{eq:thm2_detail_res} and~\eqref{eq:thm2_5} are already very close.
When $\displaystyle{\lim_{h\rightarrow 0}B(h)}=\overline{B_0} < 1$,
there exists $\lambda_4 > 0$ such that for sufficiently small $h$, we have $1 - (1+\lambda_4)^3B(h) > \frac{1}{2}(1 - \overline{B_0})$. 
Rearranging the term $E|g(X_T^\pi) - Y_T^{\pi}|^2$ in inequality~\eqref{eq:thm2_5} yields our final estimate.
\end{proof}

We shall briefly discuss how the universal approximation theorem can be applied based on Theorem \ref{thm:main2_detail}. For instance, Theorem 2.1 in \cite{arora2016understanding} states that every continuous and piecewise linear function with $m$-dimensional input can be represented by a deep neural network with rectified linear units and at most $\lceil1 + \log_2(m+1)\rceil$ depth. Now we view $Y_0$ as a target function with input $\xi$ and $E[\tilde{Z}_{t_i}|X_{t_i}^\pi,Y_{t_i}^\pi]$ as another target function with input $(X_{t_i}^\pi,Y_{t_i}^\pi)$. Since $E|Y_0|^2 < + \infty$ and $E|E[\tilde{Z}_{t_i}|X_{t_i}^\pi,Y_{t_i}^\pi]|^2 \le E|\tilde{Z}_{t_i}|^2 < + \infty$, we know that both target functions can be approximated  in the $L^2$ sense by continuous and piecewise linear functions with arbitrary accuracy. Then the aforementioned statement implies that the two target functions can be approximated by two neural networks with rectified linear units and at most $\lceil1 + \log_2(m+1)\rceil$ depth, although the width might go to infinity as the approximation error decreases to 0. Therefore, according to Theorem~\ref{thm:main2_detail}, there exist good neural networks such that the value of the objective function is small.

Note that there still exist some concerns about the result in Theorem \ref{thm:main2_detail}. First, the function $E[\tilde{Z}_{t_i}|X_{t_i}^\pi,Y_{t_i}^\pi]$  changes when $Z_{t_j}^\pi$ changes for $j < i$. Second, the function may depend on $Y_{t_i}^\pi$. Even if the FBSDEs are decoupled so that the above two concerns do not exist, we know nothing a priori about the property of $E[\tilde{Z}_{t_i}|X_{t_i}^\pi,Y_{t_i}^\pi]$. In the next theorem, we replace $E[\tilde{Z}_{t_i}|X_{t_i}^\pi,Y_{t_i}^\pi]$ with $\sigma\transpose(t_i,X_{t_i}^\pi,u(t_i,X_{t_i}^\pi))\nabla_x u(t_i,X_{t_i}^\pi)$, which can resolve these problems. However, meanwhile we require more regularity for the coefficients of the FBSDEs.
\begin{thm}\label{thm:main2'}
Suppose Assumptions \ref{assu1}, \ref{assu2}, \ref{assu3}, \ref{assu4} and the assumptions in Theorem $\ref{thm:Feynman-Kac}$ hold true. Let $u$ be the solution of corresponding quasilinear PDEs~\eqref{eq:pde} and $L$ be the squared Lipschitz constant of $\sigma\transpose(t,x,u(t,x))\nabla_x u(t,x)$ with respect to x. With the same notations of Theorem \ref{thm:main2_detail}, when $\overline{A_0}' <1$ and
\begin{align*}
\overline{B_0}' \coloneqq H_{\mathrm{min}}L\overline{A_2}e^{\overline{A_3} T}\frac{(e^{\overline{A_1}T} -1)(1- e^{-(\overline{A_1}+\overline{A_3})T})}{\overline{A_1}(\overline{A_1}+\overline{A_3})}[1 - \overline{A_0}']^{-1}(1+\lambda_3^{-1}) < 1,
\end{align*}
there exists a constant $C>0$ depending on $E|\xi|^2$, $T$, $\mathscr{L}$, $L$, $\lambda_1$, $\lambda_2$, and $\lambda_3$, such that for sufficiently small $h$,
\begin{equation}
E|g(X_T^\pi) - Y_T^\pi|^2 \le C\Big\{ h + E|Y_0 - Y_0^\pi|^2 + \sum_{i=0}^{N-1}E|f_i(X_{t_i}^\pi) - Z_{t_i}^\pi|^2h \Big\},
\label{eq:thm6_res}
\end{equation}
where $f_i(x) = \sigma\transpose(t_i,x,u(t_i,x))\nabla_x u(t_i,x)$.
\end{thm}

\begin{proof}
By Theorem \ref{thm:Feynman-Kac}, we have $Z_{t_i} = f_i(X_{t_i})$, in which $X_t$ is the solution of 
\begin{equation*}
X_t = \xi + \int_{0}^{t}b(s,X_s,u(s,X_s))\, \mathrm{d}s + \int_{0}^{t}\sigma(s,X_s,u(s,X_s))\, \mathrm{d}W_s.
\end{equation*}
Using Lemma \ref{lem3} again with $ \lambda_4>0$ gives us
\begin{align*}
E|g(X_T^\pi) - Y_T^{\pi}|^2 \le(1 + \lambda_4) H_{\mathrm{min}} 
\sum_{i=0}^{N-1}E|\delta \tilde{Z}_{t_i}|^2h + C(\lambda_4)[h+E|Y_0 - Y_0^{\pi}|^2].
\end{align*}
Given the continuity of $\sigma\transpose(t,x,u(t,x))\nabla_x u(t,x)$, we know $Z_t$ admits a continuous version. Hence the term $\tilde{Z}_{t_i}$ in $\delta \tilde{Z}_{t_i} = \tilde{Z}_{t_i} -Z_{t_i}^\pi$ can be replaced with $Z_{t_i}$, i.e.,
\begin{align}
E|g(X_T^\pi) - Y_T^{\pi}|^2 \le(1 + \lambda_4) H_{\mathrm{min}} 
\sum_{i=0}^{N-1}E|Z_{t_i} - Z_{t_i}^\pi|^2h + C(\lambda_4)[h+E|Y_0 - Y_0^{\pi}|^2].
\label{eq:thm6_1}
\end{align}
Similar to the arguments in inequalities~\eqref{eq:thm2_2}\eqref{eq:thm2_3}, we have
\begin{align*}
&~ E|Z_{t_i} - Z_{t_i}^\pi|^2 \notag \\
\leq&~ (1 + \lambda_4^{-1})E|f_i(X_{t_i}^\pi) - Z_{t_i}^\pi|^2 + (1 + \lambda_4)E|Z_{t_i} - f_i(X_{t_i}^\pi)|^2 \notag  \\
\leq&~ (1 + \lambda_4^{-1})E|f_i(X_{t_i}^\pi) - Z_{t_i}^\pi|^2 + (1 + \lambda_4)LE|X_{t_i} - X_{t_i}^\pi|^2 \notag  \\
\leq&~ (1 + \lambda_4^{-1})E|f_i(X_{t_i}^\pi) - Z_{t_i}^\pi|^2 \notag \\
&~ + (1 + \lambda_4)L
[(1 + \lambda_4)E|X_{t_i}^\pi- \overline{X}_{t_i}^\pi|^2 + (1 + \lambda_4^{-1})E|X_{t_i} - \overline{X}_{t_i}^\pi|^2 ] \notag \\
\leq&~ (1 + \lambda_4)^2LE|X_{t_i}^\pi- \overline{X}_{t_i}^\pi|^2 + C(L,\lambda_4)\Big\{E|f_i(X_{t_i}^\pi) - Z_{t_i}^\pi|^2 +  h \Big\}, 
\end{align*}
where the last equality uses the convergence result~\eqref{eq:discrete_convergece}. Plugging it into~\eqref{eq:thm6_1}, we have
\begin{align}
E|g(X_T^\pi) - Y_T^{\pi}|^2 \le
&~ (1 + \lambda_4)^3 H_{\mathrm{min}} L
\sum_{i=0}^{N-1}E|X_{t_i}^\pi - \overline{X}_{t_i}^\pi|^2h \notag \\
&~ + C(L,\lambda_4)\Big\{h + E|Y_0 - Y_0^\pi|^2 + \sum_{i=0}^{N-1}E|f_i(X_{t_i}^\pi) - Z_{t_i}^\pi|^2h \Big\} \label{eq:thm6_2}
\end{align}
for sufficiently small $h$.

We employ the estimate~\eqref{eq:overline_P_prime} again to rewrite
inequality~\eqref{eq:thm6_2} as
\begin{align}
E|g(X_T^\pi) - Y_T^{\pi}|^2 \le
&~ (1 + \lambda_4)^4 \widetilde{B}(h)E|g(X_T^\pi) - Y_T^{\pi}|^2 \notag \\
&~ + C(L,\lambda_4)\Big\{h + E|Y_0 - Y_0^\pi|^2 + \sum_{i=0}^{N-1}E|f_i(X_{t_i}^\pi) - Z_{t_i}^\pi|^2h \Big\}, \label{eq:thm6_3}
\end{align}
where 
\begin{align*}
\widetilde{B}(h) = H_{\mathrm{min}}L\overline{A_2}e^{\overline{A_3} T}\frac{ 1- e^{-(\overline{A_1}+\overline{A_3})T} }{\overline{A_1}+\overline{A_3}}[1 - \overline{A_0}']^{-1}(1+\lambda_3^{-1})\sum_{i=0}^{N-1}e^{i\overline{A_1}h}h.
\end{align*}
Arguing in the same way as that in the proof of Theorem~\ref{thm:main2_detail}, when $\tilde{B}(h)$ is strictly bounded above by $1$ for sufficiently small $h$, we can choose $\lambda_4$ small enough and rearrange the terms in inequality~\eqref{eq:thm6_3} to obtain the result in inequality~\eqref{eq:thm6_res}. 

\end{proof}
\begin{rem}
The Lipschitz constant used in Theorem~\ref{thm:main2'} may be further estimated a priori. Denote the Lipschitz constant of function $f$ with respect to $x$ as $L_{x}(f)$, and the bound of function $f$ as $M(f)$. Loosely speaking, we have
\begin{equation*}
L_x(\sigma\transpose(t,x,u(t,x))\nabla_x u(t,x)) \le 
M(\sigma)L_x(\nabla_x u) + M(\nabla_x u) [L_x(\sigma)+ L_y(\sigma)L_x(u)].
\end{equation*}
Here $L_x(u) = M(\nabla_x u(t,x))$ can be estimated from the first point of Theorem~\ref{thm4} and $L(\nabla_x u(t,x)) = M(\nabla_{xx} u)$ can be estimated through the Schauder estimate (see, e.g., \cite[Chapter 4, Lemma 2.1]{ma2007forward}). Note that the resulting estimate may depend on the dimension $d$.
\end{rem}

\subsection{Proof of Lemmas}
\begin{proof}[Proof of Lemma~\ref{lem3}]
We construct continuous processes $X_t^\pi, Y_t^\pi$ as follows. For $t \in [t_i,t_{i+1})$, let
\begin{align*}
X_t^\pi &= X_{t_i}^\pi + b(t_i,X_{t_i}^\pi,Y_{t_i}^\pi) (t-t_i) + \sigma(t_i,X_{t_i}^\pi,Y_{t_i}^\pi)(W_t - W_{t_i}), \\
Y_t^\pi &= Y_{t_i}^\pi - f(t_i,X_{t_i}^{\pi},Y_{t_i}^\pi,Z_{t_i}^{\pi})(t - t_i) +(Z_{t_i}^{\pi})\transpose(W_t - W_{t_i}).
\end{align*}
From system~\eqref{eq:FB_system1}, we see this definition also works at $t_{i+1}$. We are interested in again the estimates of the following terms
\begin{align*}
\delta X_t &=  X_t - X_t^\pi, \quad \delta Y_t =  Y_t - Y_t^\pi, \quad \delta Z_t =  Z_t - Z_{t_i}^\pi,~~~t\in[t_i,t_{i+1}).
\end{align*}
For $t\in[t_i,t_{i+1})$, let
\begin{align*}
 \delta b_t &= b(t,X_t,Y_t) - b(t_i,X_{t_i}^\pi,Y_{t_i}^\pi), \\
\delta \sigma_t &= \sigma(t,X_t,Y_t) - \sigma(t_i,X_{t_i}^\pi,Y_{t_i}^\pi), \\
\delta f_t &= f(t,X_t,Y_t,Z_t) - f(t_i,X_{t_i}^\pi,Y_{t_i}^\pi,Z_{t_i}^\pi).
\end{align*}
By definition,
\begin{align*}
\mathrm{d}(\delta X_t) &= \delta b_t\,\mathrm{d}t + \delta \sigma_t\,\mathrm{d}W_t, \\
\mathrm{d}(\delta Y_t) &= -\delta f_t \,\mathrm{d}t + (\delta Z_t)\transpose\,\mathrm{d}W_t.
\end{align*}
Then by It\^{o}'s formula, we have
\begin{align*}
\mathrm{d}|\delta X_t|^2 &= [2(\delta b_t)^T\delta X_t + \|\delta \sigma_t\|^2]\mathrm{d}t + 2(\delta X_t)\transpose \delta \sigma_t \,\mathrm{d}W_t, \\
\mathrm{d}|\delta Y_t|^2 &= [-2(\delta f_t)\transpose \delta Y_t + |\delta Z_t|^2]  \,\mathrm{d}t + 2\delta Y_t (\delta Z_t)\transpose \,\mathrm{d}W_t.
\end{align*}
Thus,
\begin{align*}
E|\delta X_t|^2 &= E|\delta X_{t_i}|^2 + \int_{t_i}^{t}E[2 (\delta b_s)\transpose \delta X_s + \|\delta \sigma_s\|^2] \,\rmd s, \\
E|\delta Y_t|^2 &= E|\delta Y_{t_i}|^2 + \int_{t_i}^{t} E[ -2 (\delta f_s)\transpose \delta Y_s  + |\delta Z_s|^2  ]\, \rmd s. 
\end{align*}

For any $\lambda_5, \lambda_6 > 0$, using Assumptions~\ref{assu1}, \ref{assu2} and the RMS-GM inequality, we have
\begin{align}
&~ E|\delta X_t|^2 \notag \\
\le &~E|\delta X_{t_i}|^2 +  \int_{t_i}^{t}[\lambda_5E|\delta X_s|^2 + \lambda_5^{-1}E|\delta b_s|^2 + E\|\delta \sigma_s\|^2] \,\mathrm{d}s \notag \\
\le &~ E|\delta X_{t_i}|^2 + \lambda_5 \int_{t_i}^{t}E|\delta X_s|^2 \,\mathrm{d}s + \int_{t_i}^{t}K(\lambda_5^{-1} + 1)|s - t_i|\,\mathrm{d}s \notag \\
&~ + \int_{t_i}^{t}[(K\lambda_5^{-1} + \sigma_x)E|X_s - X_{t_i}^\pi|^2 + (b_y\lambda_5^{-1} + \sigma_y)E|Y_s - Y_{t_i}^\pi|^2]\,\mathrm{d}s. \label{eq:lem3_1}
\end{align}
By the RMS-GM inequality, we also have
\begin{align}
E|X_s - X_{t_i}^\pi|^2 &\le (1 + \epsilon_1)E|\delta X_{t_i}|^2 + (1 + \epsilon_1^{-1})E|X_s - X_{t_i}|^2, \label{eq:lem3_2} \\
E|Y_s - Y_{t_i}^\pi|^2 &\le (1 + \epsilon_2)E|\delta  Y_{t_i}|^2 + (1 + \epsilon_2^{-1})E|Y_s - Y_{t_i}|^2, \label{eq:lem3_3}
\end{align}
in which we choose $\epsilon_1 = \lambda_6(K\lambda_5^{-1}+\sigma_x)^{-1} $ and $\epsilon_2 = \lambda_6(b_y\lambda_5^{-1} + \sigma_y)^{-1}$. The path regularity in Theorem~\ref{thm4} tells us
\begin{align}
\sup_{s\in[t_i,t_{i+1}]} (E|X_s - X_{t_i}|^2 + E|Y_s - Y_{t_i}|^2) \le Ch.
\label{eq:lem3_4}
\end{align}
Plugging inequalities~\eqref{eq:lem3_2}\eqref{eq:lem3_3}\eqref{eq:lem3_4} into~\eqref{eq:lem3_1} with simplification, we obtain
\begin{align}
E|\delta X_t|^2 &\le [1 + (K\lambda_5^{-1} + \sigma_x + \lambda_6)h]E|\delta X_{t_i}|^2 + \lambda_5 \int_{t_i}^{t}E|\delta X_s|^2 \,\mathrm{d}s \notag \\
&\hphantom{=~} +(b_y\lambda_5^{-1} + \sigma_y + \lambda_6) E|\delta Y_{t_i}|^2 h + C(\lambda_5,\lambda_6)h^2. \label{eq:lem3_delta_x}
\end{align}
Then, 
by Gr\"{o}nwall inequality, we have
\begin{align}
& E|\delta X_{t_{i+1}}|^2 \notag \\
\le~& e^{\lambda_5 h}\{[1 + (K\lambda_5^{-1} + \sigma_x +\lambda_6)h]E|\delta X_{t_i}|^2  \notag\\
& \quad~~~~+ (b_y\lambda_5^{-1} + \sigma_y + \lambda_6)E|\delta Y_{t_i}|^2h + C(\lambda_5,\lambda_6)h^2\} \notag \\
\le~& e^{A_6h}E|\delta X_{t_i}|^2 + e^{\lambda_5h}(b_y\lambda_5^{-1}+\sigma_y + \lambda_6)E|\delta Y_{t_i}|^2h + C(\lambda_5,\lambda_6)h^2 \notag \\
\le~& e^{A_6h}E|\delta X_{t_i}|^2 + A_7E|\delta Y_{t_i}|^2h + C(\lambda_5,\lambda_6)h^2, \label{eq:lem4_grownwall_x}
\end{align}
where $A_6 \coloneqq K\lambda_5^{-1} + \sigma_x + \lambda_5 + \lambda_6$, $A_7 \coloneqq b_y\lambda_5^{-1} + \sigma_y + 2\lambda_6$, and $h$ is sufficiently small.

Similarly, with the same type of estimates in~\eqref{eq:lem3_1}\eqref{eq:lem3_delta_x}, for any $\lambda_5, \lambda_6 > 0$, we have
\begin{align}
&~E|\delta Y_t|^2  \notag \\
\le &~ E|\delta Y_{t_i}|^2 + \int_{t_i}^{t}[\lambda_5E|\delta Y_s|^2 + \lambda_5^{-1}E|\delta f_s|^2 + E|\delta Z_s|^2] \,\mathrm{d}s \notag \\
\le &~E|\delta Y_{t_i}|^2 + \lambda_5 \int_{t_i}^{t}E|\delta Y_s|^2\,\mathrm{d}s + \int_{t_i}^{t} K\lambda_5^{-1}|s - t_i|\,\mathrm{d}s \notag \\
&~ + \int_{t_i}^{t}\lambda_5^{-1}[f_xE|X_s - X_{t_i}^{\pi}|^2 + K E|Y_s - Y_{t_i}^\pi|^2]\,\mathrm{d}s + (1 + f_z\lambda_5^{-1})\int_{t_i}^{t}E|\delta Z_s|^2 \,\mathrm{d}s \notag \\
\le &~[1 + (K\lambda_5^{-1} + \lambda_6)h]E|\delta Y_{t_i}|^2 + \lambda_5\int_{t_i}^{t}E|\delta Y_s|^2\,\mathrm{d}s + (f_x\lambda_5^{-1} + \lambda_6)E|\delta X_{t_i}^\pi|^2h \notag \\
&~+(1 + f_z\lambda_5^{-1})\int_{t_i}^{t}E|\delta Z_s|^2\,\mathrm{d}s + C(\lambda_5,\lambda_6)h^2. \notag
\end{align}
Arguing in the same way of~\eqref{eq:lem4_grownwall_x}, by Gr\"{o}nwall inequality, for sufficiently small $h$, we have
\begin{align*}
&E|\delta Y_{t_{i+1}}|^2 \\
\le~&e^{A_8h}E|\delta Y_{t_i}|^2 + A_9E|\delta X_{t_i}|^2 h + (1 + f_z\lambda_5^{-1} + \lambda_6)\int_{t_i}^{t}E|\delta Z_s|^2\,\mathrm{d}s + C(\lambda_5,\lambda_6)h^2,
\end{align*}
with $A_8 \coloneqq K\lambda_5^{-1} + \lambda_5 + \lambda_6$, $A_9 \coloneqq f_x\lambda_5^{-1} + 2\lambda_6$.
Choosing $\epsilon_3 = (1 + f_z\lambda_5^{-1} + \lambda_6)^{-1}\lambda_6$ and using
\begin{equation*}
\int_{t_i}^{t_{i+1}} E|\delta Z_t|^2 \,\mathrm{d}t \le (1 + \epsilon_3)E|\delta \tilde{Z}_{t_i}|^2 h + (1 + \epsilon_3^{-1})E_z^i,
\end{equation*}
where $\delta \tilde{Z}_{t_i} = \tilde{Z}_{t_i} - Z_{t_i}^\pi$ and $E_z^i = \int_{t_i}^{t_{i+1}}E|Z_t - \tilde{Z}_{t_i}|^2\,\mathrm{d}t$, we furthermore obtain
\begin{equation}
E|\delta Y_{t_{i+1}}|^2 \le e^{A_8h}E|\delta Y_{t_i}|^2 + A_9E|\delta X_{t_i}|^2 h + A_{10} E|\delta \tilde{Z}_{t_i}|^2h + C(\lambda_5,\lambda_6)(h^2+E_z^i),
\label{eq:lem4_grownwall_y}
\end{equation}
with $A_{10} \coloneqq 1 + f_z\lambda_5^{-1} + 2\lambda_6$.

Define
\begin{align*}
M_i = \max\{E|\delta X_i|^2, E|\delta Y_i|^2\}, \quad 0 \le i \le N.
\end{align*}
Combining inequalities~\eqref{eq:lem4_grownwall_x}\eqref{eq:lem4_grownwall_y} together yields
\begin{align*}
&M_{i+1} \\
\le~& (e^{\max\{A_6, A_8\}h}+ \max\{A_7, A_9\}h)M_i + A_{10}E|\delta \tilde{Z}_{t_i}|^2h + C(\lambda_5,\lambda_6)(h^2+E_z^i) \\
\le~& e^{(\max\{A_6, A_8\} + \max\{A_7, A_9\})h}M_i + A_{10}E|\delta \tilde{Z}_{t_i}|^2h + C(\lambda_5,\lambda_6)(h^2+E_z^i).
\end{align*}
Letting $A_{11} \coloneqq \max\{A_6,A_8\} + \max\{A_7,A_9\}$,
we have
\begin{equation}
M_{i+1} \le e^{A_{11}h}M_i + A_{10}E|\delta \tilde{Z}_{t_i}|^2h + C(\lambda_5,\lambda_6)(h^2 + E_z^i).
\label{eq:lem4_Mi}
\end{equation}
We start from $M_0 = E|Y_0 - Y_0^\pi|^2$ and apply inequality~\eqref{eq:lem4_Mi} repeatedly to obtain
\begin{equation}
M_N \le A_{10}e^{A_{11}T}\sum_{i=0}^{N-1}E|\delta \tilde{Z}_{t_i}|^2h + C(\lambda_5,\lambda_6)[h + E|Y_0 - Y_0^\pi|^2],
\label{eq:lem4_MN1}
\end{equation}
in which for the last term we use the fact $\sum_{i=0}^{N-1}E_z^i \le Ch$ from inequality~\eqref{eq:path_regularity}.
Note that
\begin{align*}
A_{10} &= 1 + f_z\lambda_5^{-1} + 2\lambda_6, \\
A_{11} &\le 2K + 2K\lambda_5^{-1} + \lambda_5 + 3\lambda_6.
\end{align*}
Given any $\lambda_4 > 0$, we can choose $\lambda_6$ small enough such that
\begin{align*}
(1 + f_z\lambda_5^{-1} + 2\lambda_6)e^{A_{11}T} &\le (1+\lambda_4)(1+f_z\lambda_5^{-1})e^{(2K + 2K\lambda_5^{-1} + \lambda_5)T}.
\end{align*}
This condition and inequality~\eqref{eq:lem4_MN1} together give us
\begin{align}
M_N \le~& (1+\lambda_4)(1+f_z\lambda_5^{-1})e^{(2K + 2K\lambda_5^{-1} + \lambda_5)T}\sum_{i=0}^{N-1}E|\delta \tilde{Z}_{t_i}|^2h \notag \\
&+ C(\lambda_4,\lambda_5)[h+E|Y_0 - Y_0^\pi|^2]. \label{eq:lem4_MN2}
\end{align}
Finally, by decomposing the objective function, we have
\begin{align}
&~E|g(X_T^\pi) - Y_T^{\pi}|^2 \notag \\
= &~ E|g(X_T^\pi) - g(X_T) + Y_T - Y_T^{\pi}|^2 \notag \\
\le &~ (1+(\sqrt{g_x})^{-1}) E|g(X_T^\pi) - g(X_T)|^2 + (1+\sqrt{g_x}) E|\delta Y_N|^2 \notag \\
\le &~ (g_x + \sqrt{g_x})E|\delta X_N|^2 + (1 + \sqrt{g_x})E|\delta Y_N|^2 \notag \\
\le &~ (1 + \sqrt{g_x})^2M_N. \label{eq:lem4_terminal}
\end{align}
We complete our proof by combining inequalities~\eqref{eq:lem4_MN2}\eqref{eq:lem4_terminal} and choosing $\lambda_5 = \mathrm{argmin}_{x \in \bR^{+}} H(x)$.
\end{proof}

\begin{proof}[Proof of Lemma~\ref{lem4}]
We use the same notations as in the proof of Lemma~\ref{lem1}. As derived in~\eqref{eq:lem1_delta_Y}, for any $\lambda_7 > f_z \ge 0$, we have
\begin{equation}
E|\delta Y_{i+1}|^2 \ge [1 - (2k_f+\lambda_7)h]E|\delta Y_i|^2 + (1 - f_z\lambda_7^{-1})E|\delta Z_i|^2h - f_x\lambda_7^{-1}E|\delta X_i|^2h.
\label{eq:lem4_0}
\end{equation}
Multiplying both sides of~\eqref{eq:lem4_0} by $e^{A_5ih}(e^{-A_5T}\vee 1)/(1 - f_z\lambda_7^{-1})$ gives us
\begin{align}
&~\frac{\lambda_7(e^{-A_5T}\vee 1)}{\lambda_7 - f_z}\Big\{e^{A_5ih} E|\delta Y_{i+1}|^2 - e^{A_5(i-1)h}E|\delta Y_i|^2 + e^{A_5ih} \frac{f_x}{\lambda_7}E|\delta X_i|^2h\Big\} \notag \\
\ge &~e^{A_5ih}(e^{-A_5T}\vee 1)E|\delta Z_i|^2h \notag \\
\ge &E|\delta Z_i|^2h. \label{eq:lem4_1}
\end{align}
Summing~\eqref{eq:lem4_1} up from $i=0$ to $N-1$, we obtain
\begin{align}
\sum_{i=0}^{N-1}E|\delta Z_i|^2h \le \frac{\lambda_7(e^{-A_5T}\vee 1)}{\lambda_7 - f_z} \Big\{e^{A_5T - A_5h}E|\delta Y_N|^2 + \frac{f_x}{\lambda_7}\sum_{i=0}^{N-1}e^{A_5ih}E|\delta X_i|^2h  \Big\}.
\label{eq:lem4_2}
\end{align}
Note that $E|\delta Y_N|^2 \le g_xE|\delta X_N|^2$ by Assumption~\ref{assu1}.
Plugging it into~\eqref{eq:lem4_2}, we arrive at the desired result.
\end{proof}

\begin{proof}[Proof of Lemma~\ref{lem5}]
We prove by induction backwardly.
Let $Z_{t_N}^{\pi,'} = 0$ for convenience. It is straightforward to see that the statement holds for $i = N$.
Assume the statement holds for $i = k+1$. For $i = k$, we know $Y_{t_{k+1}}^{\pi,'} = U_{k+1}(X_{t_{k+1}}^\pi, Y_{t_{k+1}}^\pi)$.
Recalling the definition of $\{X_{t_i}^\pi\}_{0\le i\le N}, \{Y_{t_i}^\pi\}_{0\le i\le N}$ in~\eqref{eq:FB_system1}, we can rewrite $Y_{t_{k+1}}^{\pi,'} = \overline{U}_{{k}}(X_{t_k}^\pi, Y_{t_k}^\pi,\Delta W_k)$, with $\overline{U}_k: \bR^m \times \bR \times \bR^d \rightarrow \bR$ being a deterministic function. Note $Z_{t_k}^{\pi,'} = h^{-1}E[\overline{U}_k(X_{t_k}^\pi, Y_{t_k}^\pi,\Delta W_k)\Delta W_k|\mathcal{F}_{t_k}]$. Since $\Delta W_k$ is independent of $\mathcal{F}_{t_k}$, there exists a deterministic function $V_k^{\pi}: \bR^m \times \bR \rightarrow \bR^d$ such that $Z_{t_k}^{\pi, '}=V_k^\pi(X_{t_k}^\pi, Y_{t_k}^\pi)$.

Next we consider $Y_{t_k}^{\pi,'}$. Let $H_k = L^2(\Omega, \sigma(X_{t_k}^\pi, Y_{t_k}^\pi), \mathbb{P})$, where $\sigma(X_{t_k}^\pi, Y_{t_k}^\pi)$ denotes the $\sigma$-algebra generated by $X_{t_k}^\pi, Y_{t_k}^\pi$. We know $H_k$ is a Banach space and another equivalent representation is 
\begin{equation*}
H_k=\{Y = \phi(X_{t_k}^{\pi}, Y_{t_k}^{\pi})~|~\phi \text{ is measurable  and } E|Y|^2 < \infty \}.
\end{equation*}
Consider the following map defined on $H_k$:
\begin{equation*}
\Phi_k(Y) = E[{Y}_{t_{k+1}}^{\pi,'} + f(t_k,{X}_{t_k}^{\pi},Y,{Z}_{t_{k}}^{\pi,'})h|\mathcal{F}_{t_k}].
\end{equation*}
By Assumption~\ref{assu3}, $\Phi_k(Y)$ is square-integrable. Furthermore, following the same argument for $Z_{t_k}^{\pi,'}$, $\Phi_k(Y)$ can also be represented as a deterministic function of $X_{t_k}^\pi, Y_{t_k}^\pi$. Hence, $\Phi_k(Y) \in H_k$. Note that Assumption~\ref{assu1} implies $E|\Phi_k(Y_1) - \Phi_k(Y_2)|^2 \le Kh^2 E|Y_1 - Y_2|^2$. Therefore $\Phi_k$ is a contraction map on $H_k$ when $h < 1/\sqrt{K}$. By the Banach fixed-point theorem, there exists a unique fixed-point $Y^*=\phi_k^*(X_{t_k}^\pi, Y_{t_k}^\pi) \in H_k$ satisfying $Y^* = \Phi_k(Y^*)$. We choose $U_k^\pi=\phi_k^*$ to validate the statement for $Y_{t_k}^{\pi,'}$.

When $b$ and $\sigma$ are independent of $y$, all of the arguments above can be made similarly with $U_i^\pi, V_i^\pi$ also being independent of $Y$.
\end{proof}

\section{Numerical Examples}
\subsection{General Setting}
In this section, we illustrate the proposed numerical scheme by solving two high-dimensional coupled FBSDEs adapted from the literature. The common setting for these two numerical examples is as follows. We assume 
$d=m=100$, that is, $X_t, Z_t, W_t \in \bR^{100}$. Assume $\xi$ is deterministic and we are interested in the approximation error of $Y_0$, which is also a deterministic scalar. 

We use $N-1$ fully-connected feedforward neural networks to parameterize $\phi_i^\pi, i=0,1,\dots,N-1$. Each of the neural networks has 2 hidden layers with dimension $d+10$. The input has dimension $d+1~(X_i\in \bR^d, Y_i \in \bR)$ and the output has dimension $d$. In practice, one can of course choose $X_i$ only as the input. We additionally test this input for the two examples and find no difference in terms of the relative error of $Y_0$ (up to second decimal place). We use the rectifier function (ReLU) as the activation function and adopt batch normalization~\cite{ioffe2015BN} right after each matrix multiplication and before activation.
We employ the Adam optimizer~\cite{kingma2015adam} to optimize the parameters with batch-size being 64. The loss function is computed based on 256 validation sample paths. We initialize all the parameters using a uniform or normal distribution and run each experiment 5 times to report the average result.

\subsection{Example 1}
The first problem is adapted from~\cite{milstein2006numerical,huijskens2016efficient}, in which the original spatial dimension of the problem is 1. We consider the following coupled FBSDEs
\begin{align}
\begin{dcases}
    X_{j,t} = x_{j,0} + \int_{0}^{t}\frac{X_{j,s}(1+X_{j,s}^2)}{(2+X_{j,s})^3}\, \mathrm{d}s \\
    \qquad\quad + \int_{0}^{t} \frac{1+X_{j,s}^2}{2+X_{j,s}^2}\sqrt{\frac{1+2Y_s^2}{1+Y_s^2+\exp{\Big(-\frac{2|X_s|^2}{d(s+5)}\Big)}}}\, \mathrm{d}W_{j,s}, \quad j=1,\dots,d,\\
    Y_t = \exp{\Big(-\frac{|X_T|^2}{d(T+5)}\Big)} \\
    \qquad~ + \int_{t}^{T}a(s,X_s,Y_s)+\sum_{j=1}^db(s,X_{j,s},Y_s)Z_{j,s}\, \mathrm{d}s - \int_{t}^{T}(Z_s)\transpose\, \mathrm{d}W_s,
\end{dcases}
\label{eq:example1}
\end{align}
where $X_{j,t}, Z_{j,t}, W_{j,t}$ denote the $j$-th components of $X_t, Y_t, W_t$, and the coefficient functions are given as
\begin{align*}
a(t,x,u) = &~ \frac{1}{d(t+5)}\exp{\Big(-\frac{|x|^2}{d(t+5)}\Big)} \\
&~ \times
\sum_{j=1}^d \Bigg\{\frac{4x_j^2(1+x_j^2)}{(2+x_j^2)^3} 
+  \frac{(1+x_j^2)^2}{(2+x_j^2)^2} 
- \frac{2x_j^2(1+x_j^2)^2}{d(t+5)(2+x_j^2)^2}
- \frac{x_j^2}{t+5}\Bigg\}, \\
b(t,x_j,u) = &~ \frac{x_j}{(2+x_j^2)^2}\sqrt{\frac{1+u^2+\exp{\Big(-\frac{|x|^2}{d(t+5)}\Big)}}{1+2u^2}}.
\end{align*}
It can be verified by It\^{o}'s formula that the $Y$ part of the solution of~\eqref{eq:example1} is given by
\begin{align*}
Y_t = \exp{\Big(-\frac{|X_t|^2}{d(t+5)}\Big)}.
\end{align*}

Let $\xi=(1,1,\dots,1)$ ($100$-dimensional), $T=5, N=160$. The initial guess of $Y_0$ is generated from a uniform distribution on the interval $[2, 4]$ while the true value of $Y_0\approx 0.81873$. We train 25000 steps with an exponential decay learning rate that decays every 100 steps, with the starting learning rate being 1e-2 and ending learning rate being 1e-5. Figure~\ref{fig:exp} illustrates the mean of the loss function and relative approximation error of $Y_0$ against the number of iteration steps. All runs converged and the average final relative error of $Y_0$ is $0.39\%$.
\begin{figure}[ht]
\centering
\includegraphics[width =6.0cm]{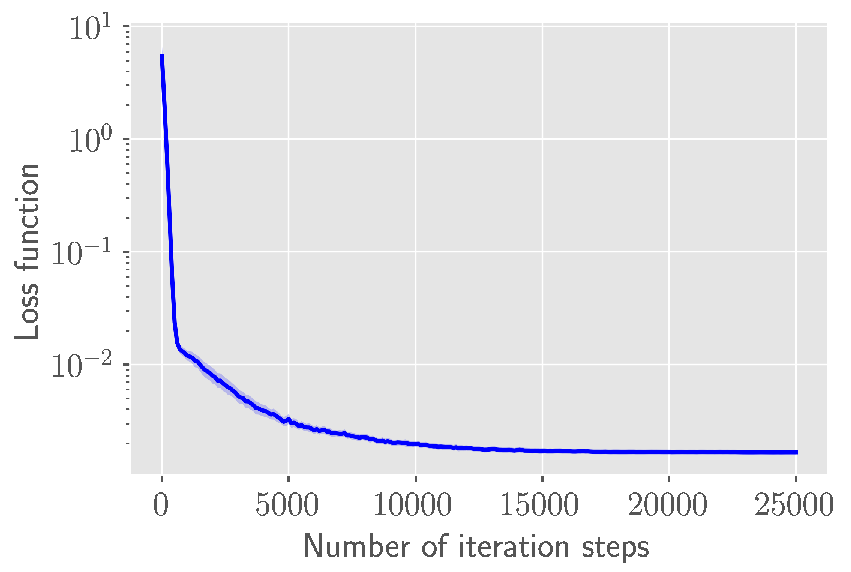}
\includegraphics[width =6.0cm]{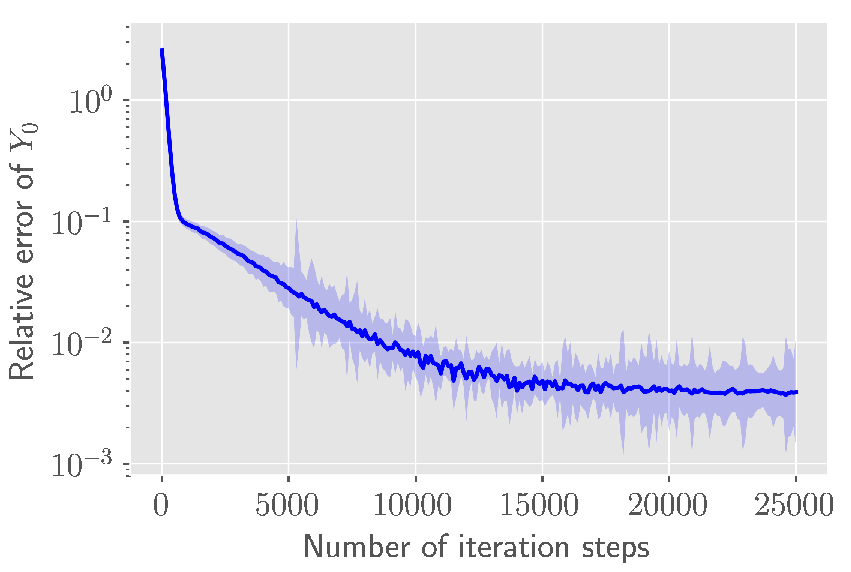}
\caption{Loss function (left) and relative approximation error of $Y_0$ (right)
against the number of iteration steps in the case of Example 1 ($100$-dimensional).
The proposed deep BSDE method achieves a relative error of size $0.39\%$. The shaded area depicts the mean $\pm$ the standard deviation of the associated quantity in 5 runs.
\label{fig:exp}}
\end{figure}

\subsection{Example 2}
The second problem is adapted from~\cite{bender2008time}, in which the spatial dimension is originally tested up to 10. The coupled FBSDEs are given by
\begin{align}
\begin{dcases}
    X_{j,t} = x_{j,0}  + \int_{0}^{t} \sigma Y_s\, \mathrm{d}W_{j,s}, \quad j=1,\dots,d,\\
    Y_t = D\sum_{j=1}^d \sin(X_{j,T}) \\
    \qquad~ + \int_{t}^{T} -rY_s + \frac12 e^{-3r(T-s)\sigma^2}
    \Big(D\sum_{j=1}^d \sin(X_{j,s})\Big)^3\, \mathrm{d}s - \int_{t}^{T}(Z_s)\transpose\, \mathrm{d}W_s,
\end{dcases}
\label{eq:example2}
\end{align}
where $\sigma>0, r, D$ are constants. One can easily check by It\^{o}'s formula that the $Y$ part of the solution of~\eqref{eq:example2} is 
\begin{equation*}
Y_t = e^{-r(T-t)}D\sum_{j=1}^d \sin(X_{j,t}).
\end{equation*}

Let $\xi=(\pi/2,\pi/2,\dots,\pi/2)$ ($100$-dimensional), $T=1, r=0.1, \sigma=0.3, D=0.1$. The initial guess of $Y_0$ is generated from a uniform distribution on the interval $[0, 1]$ while the true value of $Y_0\approx 9.04837$. We train 5000 steps with an exponential decay learning rate that decays every 100 steps, with the starting learning rate being 1e-2 and the ending learning rate being 1e-3. When $h=0.005~(N=200)$, the relative approximation error of $Y_0$ is $0.09\%$. Furthermore, we test the influence of the time partition by choosing different values of $N$. In all cases, the training converged, and we plot in Figure~\ref{fig:sin} the mean of relative error of $Y_0$ against the number of time steps $N$. It is clearly shown that the error decreases as $N$ increases ($h$ decreases).

\begin{figure}[ht]
\centering
\includegraphics[width = 9cm]{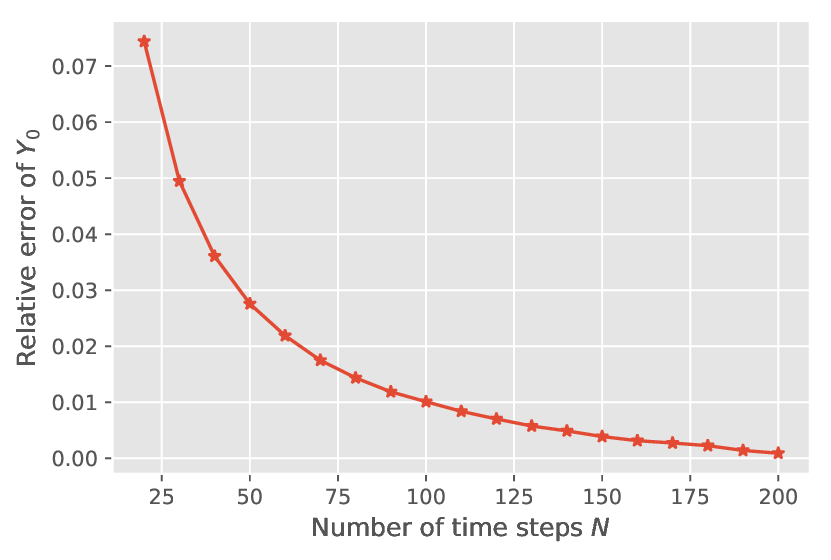}
\caption{Relative approximation error of $Y_0$
against the time step size $h$ in the case of Example 2 ($100$-dimensional).
The proposed deep BSDE method achieves a relative error of size $ 0.09\%$ when $N=200~(h=0.005)$.
\label{fig:sin}}
\end{figure}

\bibliographystyle{unsrt}
\bibliography{bib}
\end{document}